\documentclass[11pt,final]{article}
\usepackage[pdftex,colorlinks=true,urlcolor=blue,citecolor=blue]{hyperref}

\usepackage{mystuff}
\usepackage{delimdelim,showlabels,booktabs}

 \newcommand*{\doi}[1]{\href{http://dx.doi.org/\detokenize{#1}}{doi: \detokenize{#1}}}
\newcommand{\diameter}{\bar{d}}
\newcommand{\area}{\operatorname{Area}}

\newcommand{\fine}{\mathsf{f}}
\newcommand{\coarse}{\mathsf{c}}
\numberwithin{equation}{section}
 
\crefformat{ineq}{(#2#1#3)}
\crefmultiformat{ineq}{(#2#1#3)}%
{ and~(#2#1#3)}{, (#2#1#3)}{ and~(#2#1#3)}

\crefformat{subsection}{\S#2#1#3}
\crefmultiformat{subsection}{\S\S#2#1#3}{and~#2#1#3}{, #2#1#3}{  and~#2#1#3}
\crefname{figure}{Figure}{Figures}%
\usepackage{pgfplots} 
\pgfplotsset{compat=1.3}
\title{A walk outside spheres for the fractional Laplacian: fields and first eigenvalue}
\author{Tony Shardlow\footnote{Department of Mathematical Sciences,
 University of Bath, Bath BA2 7AY, UK. %
\texttt{t.shardlow@bath.ac.uk}%
}%
}
\date{\today}
 
\begin{document}
\maketitle
\begin{abstract}
	The  solution of the exterior-value problem for the
	fractional Laplacian can be computed by a walk-outside-spheres algorithm. This involves sampling \(\alpha\)-stable Levy processes on their exit from
	maximally inscribed balls and sampling their occupation
	distribution. Kyprianou, Osojnik, and Shardlow (2017) developed this
	algorithm, providing a complexity analysis and an  implementation,
	for approximating the solution at a single point in the domain. This
	paper shows how to efficiently sample the whole field by  generating
	an approximation in $L^2(D)$, for a domain $D$. The method takes
	advantage of a hierarchy of triangular meshes and uses the
	multilevel Monte Carlo method for Hilbert space-valued quantities of
	interest. We derive complexity bounds in terms of the fractional
	parameter $\alpha$ and demonstrate that the method gives accurate results
	for
	two
	problems with exact solutions. Finally, we show how to couple the
	method with the variable-accuracy Arnoldi iteration to compute the
	smallest eigenvalue of the fractional Laplacian. A criteria is
	derived for the variable accuracy and a comparison is given with
	analytical results of Dyda (2012).

	\medskip


	\medskip

\end{abstract}


\section{Introduction}
Walk On Spheres is a classical method for solving the Poisson problem
\[
	-\Delta u=f\;%
	\text{on $D$},%
	\qquad u=g\; %
	\text{on $\partial D$}
\]
for a  domain $D\subset \real^d$, boundary data $g\colon \partial
	D\to \real$ and source term $f\colon D\to \real$. In~\cite{kos}, the
algorithm was extended to a Walk Outside Spheres (WOS) algorithm for the
following problem for the fractional Laplacian: find $u\colon D\to \real$
such that
\begin{equation}
	\pp{-\Delta}^{\alpha/2} u=f\; \text{ on $D$},\qquad u=g \;\text{on
		$D^{\coarse}$},%
	\label{main}
\end{equation}
for $\alpha\in(0,2)$ and exterior data $g\colon D^{\coarse}\to\real$. A review of
the fractional Laplacian and its importance in the applied sciences
is given by~\cite{Lischke2018-qb}, which includes a description of WOS and
other
numerical
approaches. The fractional Laplacian in \cref{main} is the generator of the $\alpha$-stable
Levy process on $\real^d$. This observation leads to the following identity,
on which
the WOS algorithm is based:
\begin{equation}
	u(\vec x)%
	=\mean{g(\vec X(\tau))+\sum_{n=0}^{N^{\vec x
		}-1} r_n^\alpha \int_{B(\vec 0,1)}f(\vec x_n+r_n \vec y) \,V_1(\vec y)\,d\vec y},%
	\qquad%
	\vec x\in D.
	\label{kos}
\end{equation}
To define the terms here, consider an $\alpha$-stable Levy
process $\vec X(t)$, $t\ge 0$, starting at $\vec X(0)=\vec x$,
and let $\tau>0$ be the first-exit time of $\vec X(t)$ from $D$.
Define the discrete-time WOS process starting at $\vec x_0=\vec
	x$ by $\vec x_n=\vec X(\tau_n)$, where $\tau_0=0$ and $\tau_n=\min\Bp{t>
		\tau_{n-1}\colon \vec X(t)\not  \in B(\vec x_{n-1},r_n)}$, where $B(\vec
	x_{n-1},r_n)$ is the ball of maximum radius $r_n$ contained in $D$ and centred at
$\vec x_{n-1}$. Further, $N^\vec x$ is the exit time for the WOS process
$\vec x_n$ from $D$ (so $\tau_{N^{\vec x}}=\tau$ and $\vec X(\tau)=\vec {x}_{N^{\vec
			x}}$) and $V_1$ is the expected occupation  density of $\vec
	X(t)$ before exiting a unit ball after starting at the origin,
and given by~\cite{BGR}
\begin{align}
	V_1(\vec y)
	= 2^{-\alpha}\,\pi^{-d/2}\, \frac{\Gamma(d/2)}{\Gamma(\alpha/2)^{2}}\,
	\norm{\vec y}^{\alpha -d}\, {\int_0^{\norm{\vec y}^{-2}-1}(u+1)^{-d/2}u^{\alpha/2-1}\, du}, \quad \norm{\vec y}\le 1.\label{V}
\end{align}

\cite{kos} provides a complete Monte Carlo algorithm for
approximating $u(\vec x)$ for a single $\vec x\in D$. It works by sampling
the WOS process $\vec x_n$ and the occupation density $V_1$, and computing
$u(\vec x)$ as a sample average. \cite{kos} includes an
implementation~\cite{WOS-code} and a study of  the mean  number of WOS steps
required
for the sampling of $\vec X(\tau)$. This paper addresses the problem
of calculating the whole field $u\colon D\to \real$ as an element of
$L^2 (D)$ and the leading eigenvalue of the fractional Laplacian.
In~\cref{review}, we review  the existence and regularity
theory
for~\cref{main} as well as the Walk Outside Spheres (WOS) algorithm
from~\cite{kos}. \cref{field} develops a more efficient algorithm for
computing the solution $u\in L^2(D)$ using multilevel Monte Carlo.
The key  step is the coupling between WOS  solves on nested
triangular meshes, which leads to a bound  on the complexity
of the
method. Numerical experiments  show  that the method gives accurate results for two
problems with known solutions. In \cref{eigenvalue}, we turn to the
computation of the smallest eigenvalue of the fractional Laplacian by
using the field solve as part of an Arnoldi iteration. To make the
process more efficient, we show how the accuracy of the field solve
should be varied during the Arnoldi iteration. Computations are given,
comparing the method to analytical results of~\cite{Dyda2012-kl}. The
Julia code used for the computations is available for
download\footnote{\url{https://github.com/tonyshardlow/julia_wos}}. In \cref{conclusion}, we include a comparison of the proposed method to the adaptive finite element method of \cite{Ainsworth2017-xf}.

\section{Review}\label{review}
We will make use of the following bounds on the unique solution to
\cref{main}. We use $C^r(D)$ to denote the Banach space of $r$-times differentiable functions
with the
supremum norm on derivatives up to order $r\in \naturals$. For $s\in (0,1)$, $C^{r+s}(D)$ denotes the H\"older-continuous subspace of $u\in C^r(D)$ with norm $\norm{u}_{C^r}+\sup_{\abs{\vec r}=r}[\mathcal{D}^{\vec r} u]_s<\infty $, for
\[
	[u]_s\coloneq \sup_{\vec x,\vec y\in D } \frac{\abs{u(\vec x)-u(\vec y)}}{
		\norm{\vec x-\vec y}^s}
\]
and $\mathcal{D}^{\vec r}$ is the partial derivative defined by the multi-index $\vec r$.
\begin{theorem}
	\label[theorem]{thm:eur}
	For a  bounded Lipschitz domain $D$, suppose that $g\colon D^{\coarse}\to
		\real$
	is continuous and satisfies
	\[
		\int_{D^\coarse} \frac{\abs{g(\vec x)}}%
		{1+\norm{\vec x}^{\alpha+d}}\,d\vec x<\infty,
	\]
	and that $f\in C^r(\overline D)$, for some $r\in\naturals$ with $r>\alpha$.
	Then, there exists a unique continuous solution to \cref{main}. Further, if $g$ is uniformly bounded,
	$\norm{u}_{C^{r+\alpha}} \le C(\alpha,D) (\norm{g}_\infty + \norm{f}_{C^r})$.
\end{theorem}

\begin{proof}
	The existence and uniqueness is provided by \cite[Theorem 6.1]{kos}. This
	result also gives that $\norm{u}_\infty \le C\, ( \norm{f}_\infty+\norm{g}_\infty)$
	for a constant $C$ independent of $u$. Now apply  \cite[Corollary
		3.5]{Ros-Oton2016-nf} on any ball contained in $D$ to gain the $C^{r+\alpha}$
	regularity.
\end{proof}
\subsection{Point WOS}%
\label{pwos}
The WOS method for approximating the solution $u(\vec x)$ of \cref{main} at  a single $\vec x\in D$ is based on the probabilistic representation~\eqref{kos}. The WOS  process $\vec x_n$ can be efficiently sampled as follows: let $\mathrm{Beta}(\alpha,\beta)$ denote the  family of Beta distributions,  $S^{d-1}$
denote the unit sphere in $\real^d$, and $\mathrm{U}(S)$ denote the
uniform distribution on a bounded measurable subset $S$ of $\real^d$. Choose
independent
samples
$\beta_n\sim \mathrm{Beta}(\alpha/2, (2-\alpha)/2)$ and  $\vec{\Theta}_n\sim \mathrm{U}(S^{d-1})$, and define
the iteration
\begin{equation}
	\vec x_{n+1}%
	= \vec x_n %
	+ \vec \Theta_n%
	\frac{1}{\sqrt{\beta_n}}\,%
	d( \vec x_n),\qquad \vec x_0=\vec x,%
	\label{wos}
\end{equation}
where $d(\vec y)$ denotes the distance from $\vec y$ to the boundary of $D$.
Define the exit time $N^{\vec x}=\min\Bp{n\ge 0\colon \vec x_n\not\in D}$.
For independent samples $S_k\sim
	\mathrm{U}([0,1])$ and $\vec \Phi_k \sim \mathrm{U}(S^{d-1})$, let
\begin{equation}
	v(\vec x)%
	\coloneq g( \vec x_{N^{\vec x}})%
	+ \sum_{k=0}^{N^{\vec x}-1} F(\vec x_k; S_k, \vec \Phi_k)\label{XX}
\end{equation}
where
\[
	F(\vec x_k; S_k, \vec \Phi_k) \coloneq A_1\, d(\vec x_k)^\alpha  \bp{\pp{ f(\vec x_k+d(\vec x_k) S_k^{1/\alpha} \vec \Phi_k) -f(\vec x_k)} \prob{\beta_k<1-S_k^{2/\alpha}}+A_2\, f(\vec x_k)}
\]
for
\[
	A_1 \coloneq \frac{1}\alpha 2^{-\alpha+1} \Gamma(\alpha/2)^{-2} B((2-\alpha)/2, \alpha/2), \qquad A_2 \coloneq\int_0^1 \prob{\beta_n<1-z^{2/\alpha}}\,dz.
\]
To understand the relation to \cref{kos}, note that $\vec \Theta/\sqrt{\beta}$
has the same law as the exit distribution of the $\alpha$-stable Levy
process from a unit ball for $\beta\sim\mathrm{Beta}(\alpha/2,
	(2-\alpha)/2)$ and $\vec\Theta\sim \mathrm{U}(S^{d-1})$. The expectation
of
$A_1\, d
	(\vec x_k)^\alpha\,( f(\vec x_k+d(\vec x_k)\, S_k^{1/\alpha}\, \vec \Phi_k)$ equals
\[
	\mean{r_n^\alpha %
		\int_0^{\norm{\vec y}^{-2}-1}%
		f(\vec  x_n+r_n\vec y)\, V_1(\vec y)\,d\vec y}.
\]
The adjustment in the definition of $F$ is for computational
efficiency: we may precompute $A_2$ (with
quadrature methods), so that the term $\pp{ f(\vec x_k+d
	(\vec x_k) S_k^{1/\alpha} \vec \Phi_k) -f(\vec x_k)}$ has
lower variance, thereby yielding more easily to Monte Carlo
approximation. To approximate $u(\vec x)$, we generate $M$
\iid samples $v^j(\vec x)$ of $v(\vec x)$ and evaluate the sample mean $\frac
	1M\sum_{j=1}^M v^j(\vec x)$. The method is unbiased so that the sample
mean converges  to the exact solution $u(\vec x)$ as $M\to \infty$ and the error is described via the central-limit theorem. See~\cite{kos} for further details.

\section{WOS field solve}\label{field}
Rather than evaluate $u(\vec x)$ at a single point, we are interested in the whole field $u\in L^2(D)$.
The basic idea is to generate point estimates for a set of points  and
use interpolation to define an approximate $u\in L^2(D)$.  We examine the error when independent samples are used at each point, and then look at multilevel Monte Carlo as a method for improving efficiency by coupling point samples.
\subsection{WOS field error: independent samples}
What errors result when we generate WOS approximations to $u(\vec z^j)$ for a
set of points $\{\vec z^j\}\subset D$ and use an interpolant to approximate
$u$ in $L^2(D)$? We answer this question for $\{\vec z^j\}$ given by the
vertices of a triangular mesh for $D$ for the root-mean-square error
(the $L^2(\Omega,L^2(D))$
norm).

We work in two dimensions ($d=2$) and consider a shape-regular triangulation
$
	\mathcal{T}^h$ of a domain $D\subset
	\real^2$ with mesh-width $h$. That is, $D$ is the union of non-intersecting
triangles $\tau\in \mathcal{T}^h$ and the mesh-width $h=\max h_\tau$ for
$h_\tau$ equal to the longest edge length of $\tau\in
	\mathcal{T}^h$. Here, the shape-regular condition says there exists $C$ such
that $\area(\tau) \ge C h_\tau^d$ for all $\tau\in
	\mathcal{T}^h$. Let $I_h$ denote the interpolation operator taking values at
the vertices of the triangulation $\mathcal{T}^h$ to the piecewise-linear
interpolant. Denote the vertices of $\mathcal{T}^h$ by $\vec z^1,\dots,\vec
	z^N$. The next lemma gives a bound on the $L^2(D)$ bias error in approximating
$u$ by an interpolant with vertex data given  by a
random vector $\vec a$.

\begin{lemma}
	\label[lemma]{mf}
	Suppose  that $D$ is a bounded  polygonal domain in $\real^2$.
	Suppose that $g\colon D^{\coarse}\to \real$ is bounded and that   $f\in C^r(D)$ for $r= 2-\alpha$.
	Let $\vec a$ be an $\real^N$-valued random variable such that $\mean{\vec a}=[u(\vec z^1),\dots,u(\vec z^N)]$,  where $\vec z^i$ are the vertices of a shape-regular triangulation with mesh-width $h$. Then,
	for some constant $C(\alpha,D)>0$,
	\[
		\norm{u-\mean{I_h \vec a}}_{L^2(D)}%
		\le C(\alpha,D)\,h^{2}\, \pp{\norm{g}_\infty + \norm{f}_{C^r}}.
	\]
\end{lemma}
\begin{proof}
	As $\mean{\vec a}=[u(\vec z^1),\dots,u(\vec z^N)]$ and $\mean{I_h\vec a}=
		I_h\mean{\vec a}$, the functions $u$ and $\mean{I_h
			\vec{a}}$ agree at the vertices $\vec z^j$. The Bramble--Hilbert
	lemma (e.g.,~\cite[Theorem 4.4.20]{Brenner2008-jg}) together with the regularity given by
	\cref{thm:eur} (for $r+\alpha=2$ so that $u\in H^2(D)$) imply the result.
\end{proof}

Denoting $s_h\coloneq I_h \vec a$, the lemma says that the expectation of
the interpolant $s_h$ is close to  $u$ in the $L^2(D)$ sense, for a fine
triangulation (small mesh-width $h$). We now look at the sample errors due to
the WOS Monte Carlo method. Let $\norm{\cdot}$ denote the Euclidean distance and $\norm{\cdot}_{\mathrm{F}}$ denote the matrix Frobenius norm.
\begin{lemma}
	\label[lemma]{wn} Let $a_i$ equal the sample average of  $M$ \iid
	WOS samples of $v(\vec z^i)$ (defined in \cref{XX}), and  $\vec s_h\coloneq
		[s_h(\vec
			z^1),\dots,s_h(\vec z^N)]$ for $s_h=I_h \vec a$. Suppose that $\var( v(\vec
		z_i))\le C$, for a
	constant $C$. Then
	\[
		\mean{\frac 1N\norm{\vec s_h-\mean{\vec s_h}}^2}
		\le C\frac1M.
	\]
\end{lemma}

\begin{proof}
	$\mean{\norm{\vec s_h-\mean{\vec s_h}}^2}%
		=\norm{\cov(\vec a)}_{\mathrm{F}}$, where $\cov(\vec a)$ is diagonal with
	entries $\var( v(\vec z^i) )/M$. As there are $N$ points $\vec z^i$ and $\var (v(\vec z^i))\le C$, the result holds.
\end{proof}
We combine the two estimates, to give an $L^2$-bound in physical and probability space on the approximation error.

\begin{theorem}
	\label{3.3} Suppose that $D$ is a bounded  polygonal
	domain in two dimensions. Suppose that $g\colon D^{\coarse}\to\real$ is bounded
	and that
	\[
		\int_{D^{\coarse}} \frac{ g(\vec x)^2}{1+\norm{\vec x}^{\alpha+d}}\,d\vec x<\infty,
	\]
	and that  $f\in C^r(D)$ for some $r>\alpha$. There exists $C(\alpha,D)>0$ such
	that
	\[
		\norm{u-s_h}_{L^2(\Omega,L^2(D))}%
		\le  C(\alpha,D) \bp{h^{2} \pp{\norm{g}_\infty+\norm{f}_{C^r}}%
			+  \frac{1}{\sqrt M}},
	\]
	where $s_h$ is the piecewise-linear interpolant on a shape-regular triangulation $\mathcal{T}^h$ of $D$ of  averages of $M$ \iid WOS samples at the vertices.
\end{theorem}

\begin{proof}
	Write $u-s_h=(\mean{s_h}-s_h) +(u-\mean{s_h})$. The term $(\mean{s_h}-s_h)$ is the deviation of the interpolant $s_h$ from its mean. For linear interpolants,
	\begin{align*}
		\norm{\mean{s_h}-s_h}^2_{L^2(\Omega,L^2(D))} %
		 & = %
		\sum_{\tau \in \mathcal{T}^h}%
		\mean{    \norm{\mean{s_h}-s_h}^2_{L^2(\tau)}} %
		\le \max_{\tau \in \mathcal{T}^h}\area(\tau)%
		\norm{\mean{\vec s_h}-\vec s_h}^2.%
	\end{align*}
	The area $\area(\tau)$ of each triangle is uniformly bounded
	by $h^d$, as $h$ is the maximum edge length of any $\tau$.
	Further,~\cite[Corollary 6.5]{kos} implies that the variance of
	$v(\vec z^i)$ is bounded uniformly over $\vec z^i\in D$. Hence, for a
	constant $C$, \cref{wn} implies that
	\[
		\norm{\mean{s_h}-s_h}^2_{L^2(\Omega,L^2(D))} %
		\le C\, h^d\,%
		\frac{N}{M}.
	\]
	As the number of vertices $N\le 3 \area(D)/\min_\tau\area(\tau) =
		\order{1/h^d}$ by the
	shape-regular condition. Then, for a possibly larger constant $C$,
	\[
		\mean{    \norm{\mean{s_h}-s_h}^2_{L^2(D)} }%
		\le C\frac{1}{M}.
	\]
	The second term $u-\mean{s_h}$ is described by \cref{mf} and
	\[
		\norm{u-\mean{s_h}}_{L^2(D)}
		\le C(\alpha,D)\, h^{2}\,\pp{\norm{g}_\infty + \norm{f}_{C^r}\strutB}.
	\]
	Together the  two inequalities give that,
	for a possibly larger constant $C(\alpha,D)$,
	\begin{align*}
		\norm{u-s_h}_{L^2(\Omega,L^2(D))}%
		 & \le C(\alpha,D) \,\pp{h^{2} \,\pp{\norm{g}_\infty+\norm{f}_{C^r}\strutB}
			+  \frac{1}{\sqrt M}
		}.\qedhere
	\end{align*}

\end{proof}
For accuracy $\varepsilon$, the required number of samples grows like $1/\varepsilon^2$. The required mesh-width  $h=\order{\sqrt{\varepsilon}}$  and the number of vertices grows like $1/h^d$. Hence, the required number of vertices grows like $\varepsilon^{-d/2}$  and the total work required is $\sorder{\varepsilon^{-(2+d/2)}}$. Thus, even for $d=2$, the WOS method with independent samples requires $\order{\varepsilon^{-3}}$ work.

\subsection{Multilevel Monte Carlo}
The multilevel Monte Carlo (MLMC) method offers a practical way to reduce
computational effort in Monte Carlo runs. The idea is to introduce  nested
triangular meshes with vertices $\vec z^i_\ell$ for
$i=1,\dots,N_\ell$ on level $\ell$, and define  $v_\ell\in L^2(D)$  based on WOS approximations at $\vec z^1_\ell,\dots,\vec z^{N_\ell}_\ell$. Then,
\[
	\mean{v_\ell}=\mean{v_1}%
	+\sum_{j=1}^{\ell-1} \mean{v_{j+1}-v_j}.
\]
We show how to choose the nested triangular meshes and coupling between
samples so that
the \(v_{j+1} -v_j\)  have small variances, which allows reduced computation
time for
a given accuracy level.  Giles' complexity theorem describes the relationship
between the work required and the coupling, number of samples, and errors.

\begin{theorem}[MLMC complexity theorem]%
	\label[theorem]{ct}%
	Let $u\in L^2(D)$ and  $v_\ell$ be $L^2(D)$-valued random variables
	for $\ell=\ell_0,\ell_0+1,\ell_0+2,\dots$.
	For constants $c_i$ and $a,\beta,\gamma$, suppose that the following  hold for
	$\ell=\ell_0,\ell_0+1,\dots$.
	\begin{description}
		\item[Consistency condition] $\norm{\mean{ v_\ell}-u}_{L^2(D)}\le c_1\,
			      2^{-a \ell}$.
		\item[Coupling condition]  We have $L^2(D)$-valued random variables $v_{\ell+1}^{\fine}$
		      and $v_\ell^{\coarse}$ (the super-scripts denote fine and coarse) equal
		      in distribution to $v_\ell$ that are coupled in the sense that
		      \[
			      V_\ell \coloneq \mean {\norm{ v^{\fine}_{\ell+1}-v^{\coarse}_\ell -\mean{v^{\fine}_{\ell+1}-v^{\coarse}_\ell}}_{L^2(D)}^2}  \le c_2 \,2^{-\beta \ell}.
		      \]
		      The random variables $v_\ell^{\fine}$ and $v_\ell^{\coarse}$ are independent.
		\item[Complexity condition] The cost $C_\ell$ of computing one level-$\ell$
		      sample (i.e., of $v_1$ or $v_\ell^{\coarse}$ or $v_\ell^{\fine}$)  is bounded by $c_3\, 2^{\gamma\ell}$.
	\end{description}

	For tolerance $\varepsilon$, choose the smallest number of levels $L$ such that $c_1\, 2^{-a L}\le
		\varepsilon/2$ and choose the number of  samples on level $\ell$ as
	\[
		M_\ell=\ceil*{2 \epsilon^{-2} \sqrt{C_\ell/V_\ell}\sum_
			{\ell=\ell_0}^L\sqrt{V_\ell C_\ell}},\qquad \ell=\ell_0,\ell_0+1,\dots,L-1.
	\]
	For \iid samples $(v^i_1, v_\ell^{\fine,i}, v^{\coarse,i}_\ell)$ of $(v_1,
		v^{\fine}_\ell,
		v^{\coarse}_\ell)$, define the level-$L$ MLMC approximation
	\begin{equation}
		v^{\mathrm{ML}}%
		= \frac{1}{M_0}\sum_{i=1}^{M_0} v^i_{\ell_0} %
		+ \sum_{j=\ell_0}^{L-1}%
		\frac1{M_j}%
		\sum_{i=1}^{M_j} %
		(v^{\fine,i}_{j+1}-v^{\coarse,i}_{j}).
		\label{vML}
	\end{equation}
	Suppose that $a>1/2$ and $\beta<\gamma$.   We achieve $
		\mean{\norm{u-v^
		{\mathrm{ML}}}^2_{L^2(D)}}\le \varepsilon^2$ with $\sorder{\epsilon^{-2-(\gamma-\beta)/a}}$ work.
\end{theorem}

\begin{proof}
	It is key for this application of the complexity theorem that the quantities
	of interest are Hilbert-space valued, which is described
	in~\cite[Section 2.5]{Giles2015-fj}. We have stated the result  only for
	the
	case $a>1/2$ and $\beta<\gamma$, which is most relevant for our
	application.
\end{proof}
\subsection{Nested triangulations}\label{triangulation}
We set-up the triangulations for MLMC WOS sampling. Consider
a polygonal domain
$D\subset \real^2$ and a triangulation $\Bp{\tau_1^k}$ of $D$ with mesh-width
$h$. Let $\vec
	z_1^j$ for $j=1,\dots,N_1$ denote the vertices of the triangles. Define a
refinement by dividing each triangle into four equal parts to define a new
triangulation $\{\tau_2^k\}$ and set of vertices $\vec z_2^j$. Continue
recursively to define the vertices $\vec z_\ell^j$ at level $\ell$ of the
shape-regular triangulations $\Bp{\tau_\ell^k}$ with mesh width $2^{1-k}h$.
Denote the number of vertices on level $\ell$ by $N_\ell$. Let $I_\ell\colon
	\real^{N_\ell}\to L^2(D)$ denote the piecewise-linear interpolation operator,
which sends values at vertices $\vec z_\ell^j$ to the piecewise-linear
interpolant on the level-$\ell$ triangulation.

In the following, we will be
interested in evaluating the $L^2(D)$ norm of piecewise-linear functions. This
can be done exactly by choosing a cubature rule $Q_\tau(\phi)\approx \int_\tau
	\phi(\vec x)\,d\vec x$ of degree-of-precision two on the triangle $\tau$, and
computing
\[
	\norm{\phi}_{L^2(D)}%
	=\pp{\sum_{k} Q_{\tau_\ell^k}(\phi^2)}^{1/2},
\]
which is exact for piecewise-linear functions on the level-$\ell$ triangulation.
In two dimensions, a suitable $Q_\tau$ is defined by
\[
	Q_\tau(\phi) =\frac 13\operatorname{Area}(\tau) \bp{ \phi(\vec m_1)+\phi(\vec m_2)+\phi(\vec m_3)\strutB},
\]
where $\vec m_i$ are the midpoints of the edges of $\tau$. This means, in particular, we can write, for a piecewise-linear function $\phi$ at level-$\ell$,
\begin{equation}
	\norm{\phi}_{L^2(D)}%
	= \pp{\sum_{k}\area(\tau_\ell^k) \sum_{i,j=1}^3 a_{ij}
		\phi(\vec z_i)\phi(\vec z_j)}^{1/2},
	\label{later}
\end{equation}
where $\vec z_i$ are vertices of $\tau_\ell^k$ and the coefficients $a_{ij}>0$
(independent of level $\ell$ and triangle $\tau_\ell^k$).

To apply \cref{ct}, we require values for $a,\beta,\gamma$. At each level,
the number of points increases by a factor of $2^d$ (in dimension $d$) at
most and  $\gamma=d$. From \cref{mf},
it is clear that $a=2$. In the next two sections, we describe precisely the
coupling and give a lower bound on the value of $\beta$, which describes the
strength of the coupling.

\subsection{Coupling WOS processes}
In preparation for describing the coupling in the MLMC WOS field solver,
we study the behaviour of two WOS processes $\vec x_n$ and $\vec y_n$ generated
by the same random variables. That is,
\begin{equation}
	\vec x_{n+1}%
	=\vec x_n + d(\vec x_n)\,\vec \Theta_n\,\frac{1}{\sqrt{\beta_n}},\qquad
	\vec y_{n+1}%
	=\vec y_n + d(\vec y_n)\,\vec \Theta_n\,\frac{1}{\sqrt{\beta_n}},\label{snow}
\end{equation}
for independent samples $\vec \Theta_n\sim \mathrm{U}(S^{d-1})$ and $\beta_n\sim
	\mathrm{Beta}(\alpha/2,1-\alpha/2)$. The first result describes how the
distance $r_n\coloneq \norm{\vec x_n-\vec y_n}$ depends on $n$ and relies
on the following assumption.
%
%
\begin{assumption}
	\label[assumption]{hole}
	Suppose, for some $\lambda<1$ and $\mu>0$, that
	\[
		\mean{\pp{\frac{\norm{\vec x_1-\vec y_1}}{\norm{\vec x_0-\vec
					y_0}}}^{\mu\alpha}\,1_{\Bp{\vec x_1,\vec y_1\in D}}\,|\,\vec x_0,\vec
		y_0} \le \lambda,\qquad \text{for all $\vec x_0,\vec y_0\in D$},
	\]
	for $\vec x_1=\vec x_0+d(\vec x_0)\,\vec {\Theta}/\sqrt{\beta}$ and
	$\vec y_1=\vec y_0+d(\vec y_0)\,\vec {\Theta}/\sqrt{\beta}$, where
	$\beta\sim \mathrm{Beta}(\alpha/2,1-\alpha/2)$, and $\vec\Theta\sim
		\mathrm{U}(S^{d-1})$.
\end{assumption}
We verify
the assumption holds for a large range of $\alpha$ in
\cref{app}, by computing the expectation numerically
for $B=[0,1]^2$; it is seen to hold  with $\mu=1$ for $\alpha\le 1$ and with $\mu\approx
	0.5$ for $\alpha\le 1.8$.

\begin{lemma}
	\label[lemma]{prep}%
	Consider a domain $D$ in two dimensions.
	Let $\vec x_n,\vec y_n$ be
	coupled WOS processes (as defined by \cref{snow}) and  write
	$\norm{\vec x_n-\vec y_n}\eqcolon r_n$.
	If  \cref{hole} holds for some $\mu\le 1$, there exists $C>0$ such that,
	for $n\ge 1$,
	\begin{equation}
		\mean{r_{n+1}^{\mu\alpha}\,%
			1_{\Bp{\vec x_k,\vec y_k\in D \colon k=1,\dots,n} } \,|\,\vec x_0,\vec y_0}%
		\le C\,%
		\lambda^n\,%
		r_0^{\mu\alpha},\qquad%
		\vec x_0,\vec y_0\in D.
	\end{equation}
\end{lemma}
\begin{proof}
	In two dimensions, we may write
	\[
		(\vec x_{n+1}-\vec y_{n+1})%
		=(\vec x_n -\vec y_n)+ \frac{d(\vec x_n)-d(\vec y_n)}{\sqrt {\beta_n}}\vec{\Theta},
		\qquad
		\vec{\Theta}=
		\hat{\vec r}_n \cos \theta + \hat{\vec r}_n^\perp \sin \theta,
	\]
	for some $\theta\in(-\pi,\pi]$, where $\hat{\vec r}$ is the unit vector
	in the direction $\vec x_n-\vec y_n$ and
	$\hat{\vec r}^\perp$ is a unit vector orthogonal to $\hat{\vec r}$.  Note
	that
	$d(\vec
		x)= d(\vec y+(\vec x-\vec y))\le\norm{\vec y-\vec z+(\vec x-\vec y)}\le
		d(\vec y)+\norm{\vec x-\vec y}$, by choosing $\vec z\in\partial D$ such
	that $\norm{\vec y-\vec z}=d(\vec y)$. Then, $\abs{d(\vec x_n)-d(\vec y_n)}\le
		r_n$ and
	\[
		r_{n+1}%
		\le r_n \pp{\pp{1 + \beta_n^{-1/2}\cos \theta}^2 + \beta_n^
			{-1}\sin^2 \theta }^{1/2}
		= r_n( 1 + 2 \beta_n^{-1/2}\cos \theta + \beta_n^{-1})^{1/2}.
	\]
	The pdf of the beta distribution
	$\text{Beta}(\alpha/2,(2-\alpha)/2)$ is
	\begin{align*}
		p(\beta)%
		 & =\frac1{B(\alpha/2,(2-\alpha)/2))}\, \beta^{\alpha/2-1} \,(1-\beta)^
		{-\alpha/2}                                                             \\%
		 & =\frac1 \pi \,{\sin(\pi\alpha/2)} \,\beta^{\alpha/2-1}\,(1-\beta)^
		{-\alpha/2},\qquad \beta\in(0,1).
	\end{align*}
	The term $( 1 + 2 \beta_n^{-1/2}\cos \theta + \beta_n^{-1})^{\mu \alpha/2}
		p(\beta_n)$ grows like $\beta_n^{-\mu \alpha/2}\beta_n ^{\alpha/2-1}$ as
	$\beta_n\downarrow 0$, and
	hence
	$\mean{( 1 + 2 \beta_n^{-1/2}\cos \theta + \beta_n^{-1})^{\mu \alpha/2}}=\infty$
	for $\mu> 1$. This explains the restriction to $\mu\le 1$.

	For $\mu\le 1$, let
	\[
		C\coloneq \mean{( 1 + 2 \beta_n^{-1/2}\cos \theta + \beta_n^{-1})^{\mu \alpha/2}\,
			p(\beta_n)}<\infty.
	\]
	Consequently,
	\[
		\mean{r^{\mu\alpha}_{n+1} \,|\, \vec
			x_n,\vec
			y_n}%
		\le C\, r_n^{\mu\alpha}\,1_{\Bp{\vec x_{k},\vec y_k\in D\colon k=1,\dots,n}}.
	\]
	To estimate the right-hand side in terms of the initial separation $r_0$, we use
	\cref{hole}, which provides a uniform bound on the $\mu\alpha$-moment of
	$r_{n+1}/r_n$
	given $\vec x_{n+1},\vec y_{n+1}\in D$,
	uniformly over any choice of $\vec x_n,\vec y_n \in D$ with separation
	$r_n$.
	This implies that
	\[
		\mean{ r_n^{\mu \alpha}\,1_{\Bp{\vec x_{k},\vec y_k\in D\colon k=1,\dots,n}}}
		\le \lambda\, \mean{ r_{n-1}^{\mu \alpha}\,1_{\Bp{\vec x_{k},\vec y_k\in D\colon
			k=1,\dots,n-1}}}.
	\]
	Iterating this, we achieve the result.
\end{proof}

We make precise the difficulty with the case $\mu>1$.
\begin{lemma}
	Let $D$ be the unit ball in two dimensions.	For all
	$\mu>1$ and
	$K>1$, there exists $\vec
		x_0,\vec y_0\in D$ such that
	\[
		\mean{r_{1}^{\mu\alpha}\,1_{\Bp{\vec x_1,\vec y_1\in D}}}\ge K\, r_0^{\mu\alpha}.
	\]
\end{lemma}
\begin{proof}
	Again, write
	\[
		(\vec x_1-\vec y_1)%
		=(\vec x_0-\vec y_0)%
		+ \frac{d(\vec x_0)-d(\vec y_0)}{\sqrt {\beta_0}}\pp{
			\hat{\vec r}_0 \cos \theta + \hat{\vec r}_0^\perp \sin \theta}.
	\]
	Choose $\vec z\in \partial D$ with inward-pointing normal
	$\vec n$ and let
	$L$ be the distance from $\vec z$ to
	the far boundary in the direction $\vec n$. Choose $\delta>0$ and $\vec
		x_0,\vec y_0\in D$ with $r_0=\norm{\vec x_0-\vec y_0}\le \delta$ on the line
	through $\vec z$ in direction $\vec
		n$, such that $d(\vec x_0)=\delta$ and  $d(\vec y_0)=d(\vec
		x_0)+r_0.$  By taking $\theta_0>0$ small, we can choose an interval
	$(-\theta_0,\theta_0)$ so all directions $\vec
		\Theta$ from
	$\vec x_0,\vec y_0$ that are angle $\theta\in(-\theta_0,\theta_0)$ from
	$\vec n$ are at
	least $L/2$ from the far boundary (uniformly in $\delta$
	small).
	All jumps  from
	$\vec x_0,\vec y_0$ in the direction $\theta\in(-\theta_0,\theta_0)$ of
	length $L/2$ or less remain in $D$. When
	$d(\vec x_0)/\sqrt{\beta_0}\le L/2$, we have $\beta_0\in(2 \delta/L)^2,1)$.
	Then,
	\[
		\mean{r_{1}^{\mu \alpha}} \ge r_0^{\mu\alpha} \int_{(2 \delta/L)^2}^1 \pp{1+
			\frac{1}{\beta_0^
				{1/2}} \min_{\theta\in(-\theta_0,\theta_0)}\cos \theta}^{\mu\alpha} p
		(\beta_0)\,d\beta_0.
	\]
	The pdf of the beta distribution
	$\text{Beta}(\alpha/2,(2-\alpha)/2)$ is
	\[
		p(\beta) =\frac1{B(\alpha/2,(2-\alpha)/2))}\, \beta^{\alpha/2-1} \,(1-\beta)^ {-\alpha/2} =\frac1 \pi \,{\sin(\pi\alpha/2)} \,\beta^{\alpha/2-1}\,(1-\beta)^ {-\alpha/2}.
	\]
	Consequently, for $\mu>1$,
	\[
		\int_0^1 \beta^{-\mu \alpha/2} p(\beta)\,d\beta=\infty.
	\]
	Hence,
	\[	\int_{(2 \delta/L)^2}^1 \pp{1+
			\frac{1}{\beta_0^
				{1/2}} \min_{\theta\in(-\theta_0,\theta_0)}\cos \theta}^{\mu\alpha} p
		(\beta_0)\,d\beta_0\to \infty\qquad \text{as $\delta\downarrow 0$.}
	\]
	Thus, for any $K>1$, we can choose $\delta,r_0$ (and hence $\vec x_0,\vec
		y_0$) such
	that $\mean{r_1^{\mu\alpha}\,1_{\Bp{\vec x_1,\vec y_1\in D}}}\ge K r_0^\alpha$.
\end{proof}

Via Chebyshev`s inequality, \cref{hole} implies the following bound on
the probability coupled WOS paths
are separated.

\begin{corollary}
	\label[corollary]{muni}%
	Under the assumptions of \cref{prep}, there exists a constant \(C>0\)
	and \(\lambda\in(0,1)\)
	such that, for any $\vec x_0,\vec y_0\in D$,
	\begin{equation*}
		\prob{r_n> \epsilon;\;%
			\text{\(\vec x_k,\vec y_k\in D\) for \(k=1,\dots,n-1\)}%
			\,|\, \vec x_0,\vec y_0}%
		\le C \,\frac{{ r_{0}^{\mu\alpha}}} {\epsilon^{\mu\alpha} }
		\,\lambda^n.
	\end{equation*}
\end{corollary}

\begin{proof}
	As ${r_n}\le (1+1/\sqrt \beta)\, {r_{n-1}}$, the event that $r_n>\epsilon$
	given $r_{n-1}$ is contained in
	$\sqrt{\beta} \le 1/(\epsilon/r_{n-1} -1)\le 2(r_{n-1}/\epsilon)$ for $r_{n-1}<\epsilon/2$. The density $p(\beta)$ of $\mathrm{Beta}(\alpha/2,1-\alpha/2)$ satisfies $p(\beta)=\sorder{\beta^{\theta/2-1}}$ as $\beta\downarrow 0$. Hence, for some $C>0$, $p(\beta)\le C \beta^{\alpha/2-1}$ for $\beta<1/2$. Hence, for larger constant $C$,
	\begin{equation*}
		\prob{\beta \le \frac1{(\epsilon/r_{n-1}-1)^2}\;|\; r_{n-1}} %
		\le \prob{\beta \le 4\frac{r_{n-1}^2}{\epsilon^2}\;|\; r_{n-1}}%
		\le C\,\frac{r_{n-1}^{\alpha}}{\epsilon^{\alpha}},\qquad \text{ if $\frac{2r_{n-1}}{\varepsilon}\le \frac 12$.}
	\end{equation*}
	This inequality also holds for $r_{n-1}\ge \epsilon/2$ by making sure
	$C\ge
		2^\alpha$. As $\mu\le 1$, to complete the proof, we apply \cref{prep}.
\end{proof}

We  require the following assumption regarding the WOS process near to the
boundary of $D$, to control how WOS particles accumulate near the
boundary. It is  shown to hold for a specific domain in \cref{app}
by numerically evaluating integrals. For $B=[0,1]^2$, it is seen to hold for $t\approx 1$
for $\alpha\le 0.5$ and for $t\approx 0.9$ for $\alpha\le 1.8$.

\begin{assumption}
	\label[assumption]{skin2}%
	Suppose, for some $\lambda,t\in(0,1)$
	and $A>0$,
	that
	\[
		\mean{\Phi(\vec x_1) \,|\,\vec x_0}%
		\le \lambda\, \Phi(\vec x_0),
	\]
	where
	\begin{equation}
		\Phi(\vec x)%
		\coloneq 1_{\Bp{\vec x\in D}}
		\,\max\Bp{A,\frac{1}{d(\vec x)^t}}.\label{eqq2}
	\end{equation}

\end{assumption}
The initial set of vertices has the following moment
property with respect to $\Phi(\vec x)$.
\begin{lemma}
	\label[lemma]{init}
	Let $\vec z_\ell^k$ for $i=1,\dots,N_\ell$ be the vertices of the
	triangulation at level-$\ell$ defined in \cref{triangulation}. For the $\Phi$ in
	\cref{skin2},
	there exists $C>0$ independent of $\ell$ such that
	\[
		\frac1N \sum_{k=1}^{N_\ell}%
		\Phi(\vec z_\ell^k)<C.
	\]
\end{lemma}
\begin{proof}
	Let $\diameter=\sup\Bp{\norm{\vec x-\vec y}\colon \vec x,\vec y\in D}$ denote
	the diameter of $D$.
	The vertices are distributed uniformly and the
	proportion of the $N_\ell$ vertices $\vec z^k_\ell$ such that $d(\vec
		z^k_\ell)\in ((j-1)\epsilon,  j\epsilon)$, for $j=1,\dots,\ceil{\diameter/2\varepsilon}$,
	is less than $C\epsilon$ for
	a
	constant $C$ (that depends on the geometry of $D$; $C=4/\diameter $ when
	$D$ is a
	ball of diameter $\diameter $). Hence, for $t\in(0,1)$,
	\begin{align*}
		\frac1{N_\ell} \sum_{k=1}^{N_\ell}%
		\frac{1}{d(\vec z^k_\ell)^t}%
		 & \le \sum_{j=1}^{\ceil{\diameter/2\varepsilon}} \frac{C\epsilon}{(j\epsilon)^t}
		\le C\int_0^{\diameter /2} \frac{1} {x^t}\,dx%
		=\frac C {1-t}\bp{x^{1-t}\strutB}_0^{\diameter /2}%
		= \frac C {1-t} \pp{\frac{\diameter }{2}}^{1-t}.
	\end{align*}
\end{proof}
\subsection{Main theorem on coupling}
The MLMC estimator $v^{\mathrm{ML}}$ defined in \cref{vML} is defined in
terms of $v^{\coarse}_\ell, v^{\fine}_\ell, v_{\ell_0}$. Key to the success of the estimator
is the coupling between the fine-level estimator $v^{\fine}_{\ell+1}$ and coarse-level
estimator $v^{\coarse}_\ell$, which must be  generated
by the same random variables. To write this down precisely, consider independent
random variables
$\beta_n\sim
	\mathrm{Beta}(\alpha/2,1-\alpha/2)$, $\vec\Theta_n,\vec \Phi_n\sim \mathrm{U}(S^{d-1})$,
and $S_k\sim \mathrm{U}([0,1])$. If $\vec x_n$ is the WOS process starting
at $\vec x_0=\vec x$ generated by these inputs, define
\begin{equation}
	V(\vec x)%
	=g( \vec x_{N^{\vec x}})%
	+ \sum_{k=0}^{N^{{\vec x}}-1} F(\vec x_k),\label{pkos}
\end{equation}
using $F(\vec x_k)=F(\vec x_k; S_k, \vec \Phi_k)$ (see  \cref{XX}). The
field $V$ is not easily evaluated as it requires a WOS process for every
$\vec
	x$ and we evaluate it only at the vertices of the triangulations. Let $V_\ell$
be independent copies of $V$. The fields $v^{\fine}_{\ell+1}$
and $v^{\coarse}_\ell$ are defined as linear interpolants using vertices of the
relevant triangulation as initial data. That is,
\[
	v^{\fine}_{\ell+1}%
	\coloneq I_{\ell+1} \{V_\ell(\vec z^k_{\ell+1})\colon k=1,\dots,N_{\ell+1}\},\qquad
	v^{\coarse}_{\ell}\coloneq I_\ell \{V_\ell(\vec z^k_\ell)\colon k=1,\dots,N_\ell\}.
\]
In this way, $v_{\ell+1}^{\fine}$ and $v^{\coarse}_\ell$ are given by the linear interpolant
of the same copy of $V$, based on vertices of the level $\ell+1$ or $\ell$
triangulation. Note that $v_\ell^{\fine}$ and $v_{\ell+1}^{\fine}$ are independent
(and similarly for the coarse versions).

We now give the main result on coupling
between $v^{\coarse}_\ell$ and $v^{\fine}_{\ell+1}$.
\begin{theorem}
	\label[theorem]{tcouplingrate} Suppose that \cref{hole,skin2} hold with exponents
	$t,\mu<1$.  Suppose that $f$ and
	$g$  are uniformly $\mu\alpha$-H\"older continuous with H\"older
	constant $L$. Then, the coupling condition holds for the MLMC
	complexity theorem (\cref{ct}) with $\beta=\min\Bp{\mu\alpha,\mu
			\alpha t/(t+\mu\alpha)}$. In particular, there exists $C>0$ such that
	\begin{align}
		\mean{
			\norm{v_{\ell+1}^{\fine}-v_\ell^{\coarse}-\mean{v_{\ell+1}^{\fine}-v_\ell^{\coarse}}}^2}%
		\le C\, \bp{L \,h_\ell^{\mu\alpha}%
			+  h_\ell^{\mu t\alpha/(t+\mu\alpha)} }.%
	\end{align}
\end{theorem}
\begin{proof}
	The constant $C$ in this proof is a generic constant independent
	of $\ell$ and $h_\ell$ and changes from line to line.
	Following \cref{later}, the key observation is that
	\[
		\mean{\norm{v_{\ell+1}^{\fine}-v_\ell^{\coarse}-\mean{v_{\ell+1}^{\fine}-v_\ell^{\coarse}}}_{L^2(D)}^2}%
		=\mean{\sum_{k}\operatorname{Area}(\tau_k^{\ell+1}) %
			\sum_{i,j=1}^3 a_{ij}%
			\,{\delta}(\vec z^i)%
			\,{\delta}(\vec z^j)},
	\]
	where $\vec z^i$ are the vertices of $\tau_k^{\ell+1}$ and ${\delta}=v_
		{\ell+1}^{\fine}-v_\ell^{\coarse}-\mean{v_{\ell+1}^{\fine}-v_\ell^{\coarse}}$. As
	$ab\le \tfrac12 (a^2+b^2)$, there exists $C>0$ such that
	\begin{align}
		\mean{\norm{v_{\ell+1}^{\fine}-v_\ell^{\coarse}-\mean{v_{\ell+1}^{\fine}-v_\ell^{\coarse}}}_{L^2(D)}^2}%
		 & \le C\,\max_k\operatorname{Area}(\tau^{\ell+1}_k) \, %
		\mean{\sum_{i}\delta(\vec z^i)^2} \notag                \\ %
		 & \le C\,\frac{1}{N^\ell} \,                           %
		\mean{\sum_{i}\delta(\vec z^i)^2},\label[ineq]{one}
	\end{align}
	where the sum is over all vertices $\vec z^i$ of the fine
	triangulation $\tau^
		{\ell+1}_k$. If $\vec z^i$ is also a vertex in the coarse triangulation
	$\tau^\ell_k$, $\delta(\vec z^i)=0$. Otherwise,  $\vec z^i$ is the midpoint of
	an edge $\vec y_a\leftrightarrow\vec y_b$ of the coarse triangulation. Then,
	$v_\ell^{\coarse}(\vec z^i)$ is defined by linear interpolation and
	\[
		\delta(\vec z^i)%
		=v^{\fine}_{\ell+1}(\vec z^i)%
		-\frac12 \pp{v^{\coarse}_\ell(\vec y_a)+v^{\coarse}_\ell(\vec y_b)}%
		-\mean{v^{\fine}_{\ell+1}(\vec z^i)-\frac12 \pp{v_\ell^{\coarse}(\vec y_a)+v_\ell^{\coarse}(\vec
				y_b)}}.
	\]
	As $(a+b)^2\le 2a^2+2b^2$,
	\begin{align}
		\mean{ \delta(\vec z^i)^2}  %
		 & \le \mean{(v_{\ell+1}^{\fine}(\vec z^i)-\frac12 \pp{v^{\coarse}_\ell(\vec y_a)+v^{\coarse}_\ell(\vec
		y_b)})^2} \notag                                                                                        \\%
		 & \le \frac12 \pp{\mean{(v_{\ell+1}^{\fine}(\vec z^i)-v_{\ell+1}^{\fine}(\vec y_a))^2}
		+\mean{(v_{\ell+1}^{\fine}(\vec
			z^i)- v_{\ell+1}^{\fine}
			(\vec
			y_b))^2}\strutB}.\label[ineq]{two}%
	\end{align}
	In the last line, we use the fact that $v_{\ell+1}^{\fine}$ and $v_\ell^{\coarse}$ agree
	on the level-$\ell$ triangulation.
	Putting together \cref{one,two}, this means we  establish the result if
	we show, for
	a constant $C$,
	that, for any $\vec z^k_\ell+\vec r_0\in D$,
	\begin{equation}
		\frac{1}{N_\ell}\sum_{k=1}^{N_\ell}	%
		\mean{\pp{v^{\fine}_{\ell+1}(\vec z^k_\ell)-v^{\fine}_{\ell+1}(\vec
		z^k_\ell+\vec r_0)}^2}%
		\le C\norm{\vec r_0}^\beta. \label[ineq]{get}
	\end{equation}

	We start by estimating the contribution to $v^{\fine}_{\ell+1}$ from the exterior
	function
	$g$, for paths starting at $\vec x_0$ and $\vec y_0=\vec x_0+\vec r_0$, first
	ignoring the contribution from the right-hand-side term $f$.  At the end of
	the proof, we briefly discuss what changes need to made to account for a non-zero  $f$.

	Let $N^{\vec x}$ denote the WOS exit time of a path starting from $\vec x\in D$.
	Denote coupled WOS paths by $\vec x_n$ and $\vec y_n$ with  $\vec y_0=\vec
		x_0+\vec r_0$.
	Write
	\begin{align*}
		 & \qquad \mean{\pp{g(\vec x_{ N^{\vec x}})-g(\vec y_{N^{\vec y}})}^2} %
		=\mean{\sum_
		{k\ge 1} \pp{g(\vec x_{k})-g(\vec y_k)}^2
		1_{\Bp{k=N^{\vec x}=N^{\vec y}}}}
		\\%
		 & \quad+\mean{\sum_{k\ge  1} \pp{g(\vec x_{N_{\vec x}})-g
			(\vec y_{N_{\vec y}})}^2  1_{\Bp{k=N^{\vec x}<N^{\vec y}}}}%
		+\mean{\sum_{k\ge 1} \pp{g(\vec x_{N_{\vec x}})-g
			(\vec y_{N_{\vec y}})}^2  1_{\Bp{k=N^{\vec y}<N^{\vec x}}}}.%
	\end{align*}

	For the first term, using the H\"older regularity of $g$, note that
	\begin{align*}
		\mean{
		\pp{\sum_{k\ge 1}
		\pp{g(\vec x_k)-g(\vec y_k)}^2
		1_{\Bp{k=N^{\vec x}=N^{\vec y}}}  }}%
		 & \le \sum_{k\ge 1}
		\mean{  (g(\vec x_k)-g(\vec y_k))^2%
		1_{\Bp{k=N^{\vec x}=N^{\vec y}}}}
		\\%
		 & \le    L (2\norm{g}_\infty)^{2-\mu\alpha} \sum_{k\ge 1}
		\mean{ r_{k} ^{\mu\alpha}1_{\Bp{k=N^{\vec x}=N^{\vec y}}}}.
	\end{align*}
	By \cref{prep},
	\[
		\mean{ r_{k} ^{\mu\alpha} 1_{\Bp{k=N^{\vec x}=N^{\vec y}}}}
		\le C\,
		\lambda^k\,r_0^{\mu\alpha}.
	\]
	Consequently,
	\[
		\mean{
		\pp{\sum_{k=N^{\vec x}=N^{\vec y}}
		\pp{g(\vec x_k)-g(\vec y_k)}^2  }
		1_{\Bp{N^{\vec x}=N^{\vec y}=k}}  } \le L\,(2\norm{g}_\infty)^{2-\alpha} \,r_0^{\mu\alpha} \,\sum_{k\ge 1} \lambda^k .%
	\]
	As
	$\lambda<1$, we find a $C>0$ such that
	\[
		\mean{
		\pp{\sum_{k\ge 1}
			\pp{g(\vec x_k)-g(\vec y_k)}^2  }
		\,1_{\Bp{k=N^{\vec x}=N^{\vec y}}}  }%
		\le C\, L\,r_0^{\mu\alpha}.
	\]

	It remains to consider the case where the paths exit at different times. The second and third terms are equivalent, and we consider only
	\[
		\sum_{k\ge 1} \mean{\pp{g(N^{\vec x})-g(N^{\vec y})}^2 1_{\Bp{k=N^
		{\vec y}<N^{\vec
		x}}}} \le 4\norm{g}^2_\infty  \sum_{k\ge 1}  \prob{ k=N^{\vec y}<N^{\vec x}}.
	\]
	Now,
	\begin{equation}
		\label{cook}\prob{ k=N^{\vec y}<N^{\vec x}}
		=\prob{d(\vec x_k)>\varepsilon,\;  k=N^{\vec y}<N^{\vec x}}+%
		\prob{d(\vec x_k)\le \varepsilon,\;  k=N^{\vec y}<N^{\vec x}}.%
	\end{equation}
	For the first term,  note that
	$ \prob{d(\vec x_k)>\varepsilon,\;  k=N^{\vec y}<N^{\vec x}}%
		\le \prob{r_k>\varepsilon,\,k=N^{\vec y}<N^{\vec x}}$, so, by
	\cref{muni},
	\[
		\prob{ d(\vec x_k)>\varepsilon, \,k=N^{\vec y}<N^{\vec x}}%
		\le C\,\lambda^n\,\frac{r_0^{\mu\alpha} }{\epsilon^{\mu\alpha}}.
	\]
	As $\lambda<1$,
	\[
		\sum_{k\ge 1}  \prob{ d(\vec x_k)>\varepsilon, \,k=N^{\vec y}<N^{\vec
		x}} \le C\,\frac{\lambda}{1-\lambda}\, r_0^{\mu\alpha} \,\frac 1
			{\varepsilon^{\mu\alpha}}.
	\]

	For the second term in \cref{cook}, use \cref{skin2}, to
	see that
	\[
		\mean{\Phi(\vec x_1) \,|\,\vec
			x_0}%
		\le \lambda%
		\,    \Phi(\vec x_0),\qquad \vec x_0\in D.
	\]
	Iterating, we have
	\begin{equation}
		\mean{\Phi(\vec x_n)\,1_{\Bp{\vec x_1,\ldots,\vec
			x_n\in D}} \,|\,\vec x_0}%
		\le \lambda^n
		\,\Phi(\vec x_0).\label[ineq]{seagull}
	\end{equation}
	Hence,
	\[
		\sum_{n\ge 1}
		\mean{\Phi(\vec x_n) 1_{\Bp{\vec x_1,\ldots,\vec
						x_n\in D}} \,|\,\vec x_0}%
		<\frac{1}{1-\lambda}\,\Phi(\vec x_0).
	\]
	Note that $\Phi(\vec x)\ge 1/d(\vec x)^t$ if $d(\vec x)^t\le 1/A$. Hence $\prob{d(\vec x_n)<\varepsilon}\le C\, \varepsilon^t \,\mean{\Phi(\vec
			x_n)}$ for a constant $C$. Hence,
	\[
		\sum_{k\ge 1} \prob{ d(\vec x_k)\le \varepsilon,\;N^{\vec x}\ge
		k} \le C\, \varepsilon^t \,\Phi(\vec x_0).
	\]
	To match the two probabilities and choose $\varepsilon$ in terms of
	$r_0$, put $r_0^{\mu\alpha}/{\varepsilon^{\mu\alpha}} %
		= \varepsilon^t$ so that  \(
	\varepsilon%
	=r_0^{\mu\alpha/(t+\mu\alpha)}
	\).  Thus,
	\[
		\sum_{k\ge 1} \mean{\pp{g(N_{\vec x})-g(N_{\vec y})}^2  1_{\Bp{N^{\vec
				x}>N^{\vec y}=k}}}%
		\le C\, r_0^{\mu\, t\,\alpha/(t+\mu\alpha)}\, \Phi(\vec x_0).
	\]
	Then, adding
	the contributions for all terms, we conclude that
	\begin{align}
		\mean{\pp{g(\vec x_{ N^{\vec x}})-g(\vec y_{N^{\vec y}})}^2}%
		 & \le C\pp{ L \,r_0^{\mu\alpha} %
		+  \norm{g}^2_\infty\,r_0^{\mu\,t\,\alpha/(t+\mu\alpha)}
		\,\Phi(\vec x_0)}.  \label[ineq]{am1}%
	\end{align}
	By \cref{init}, $\frac1{N_\ell}
		\sum_{k=1}^N \Phi(\vec z^k_\ell)$  is less than a $C$ independent of $\ell$.
	This
	implies \cref{get}, as we can average over initial vertices $\vec x_0=\vec
		z^k_\ell$ for $k=1,\dots,N_\ell$.  Hence, in the case $f=0$,
	\[
		\norm{v^\fine_{\ell+1}-v^\coarse_\ell-\mean{v^\fine_{\ell+1}-v^\coarse_\ell}
		}_{L^2(D)}%
		\le C\pp{ L \,r_0^{\mu\alpha}%
		+  \norm{g}^2_\infty\,r_0^{\mu t\alpha/(t+\mu\alpha)}}.
	\]

	We now discuss extending the argument to $f\ne 0$. As $f$ is
	$\mu\alpha$-H\"older continuous, from \cref{pwos}, we see that $F$ is
	$\mu\alpha$-H\"older continuous. Now,
	\begin{align*}
		I_1 & \coloneq		\mean{\pp{\sum_{k\ge 0}
		(F(\vec x_{ k})-F(\vec y_{k})) 1_{\Bp{k<\min\Bp{N^{\vec x},N^{\vec y}}}}}^2}           \\%
		    & \le\mean{\sum_{k=1}^\infty                                                       %
		1_{\Bp{k=\min \Bp{N^{\vec x}, N^{\vec y}}}}%
		\pp{\sum_{j=0}^{k-1}(F(\vec x_{ j})-F(\vec y_{j})) }^2
		}                                                                                      \\
		    & \le\mean{\sum_{k=1}^\infty                                                       %
		1_{\Bp{k=\min \Bp{N^{\vec x}, N^{\vec y}}}}%
		\sum_{j=0}^{k-1} L \,(2\norm{F}_\infty)^{2-\mu\alpha} \,r_0^{\mu \alpha} \,\lambda^j
		}                                                                                      \\
		    & \le  L\,(2\norm{F}_\infty)^{2-\mu\alpha} \,r_0^{\mu \alpha} \frac{1}{1-\lambda}.
	\end{align*}
	Hence, for a constant $C$,
	\begin{equation}
		I_1 \le C\, L\,r_0^{\mu\alpha}.\label[ineq]{a0}
	\end{equation}
	Similarly,
	\begin{align}
		\mean{\pp{\sum_{N^{\vec y}\le k<N^{\vec x}}\,F(\vec x_{k})}^2}%
		 & \le  C\, \norm{F}_\infty^2 %
		\,r_0^{t\mu\alpha/(t+\mu\alpha)}  \Phi(\vec x_0).\label[ineq]{a1}
	\end{align}

	To put everything together, recall that
	\[
		v(\vec x) = g(\vec x _{N^\vec x}) +    \sum_{k=0}^{N^{\vec x}-1} F(\vec x_k).
	\]
	Hence,
	\begin{align*}
		\abs{v^\fine_{\ell+1}(\vec x)-v^{\fine}_{\ell+1}(\vec y)}
		 & \le
		\abs{g(\vec x_{N^{\vec x}})-g(\vec y_{N^{\vec y}})}
		+\sum_{k=0}^{\min\Bp{N^{\vec x},N^{\vec y}}-1}
		\abs{F(\vec x_k)-F(\vec y_k)}               \\
		 & \quad+\sum_{k=N^{\vec x}}^{N^{\vec y}-1}
		\abs{F(\vec y_k)}%
		+\sum_{k=N^{\vec y}}^{N^{\vec x}-1}
		\abs{F(\vec x_k)}.
	\end{align*}
	The required bound \cref{get}  follows  from \cref{am1,a0,a1,seagull}.
\end{proof}
We  now summarise the implications of the MLMC complexity theorem. We have designed an MLMC algorithm for solving the exterior-value problem for the fractional Laplacian with the following complexity.
\begin{corollary}\label[corollary]{cor:complexity}
	Let the assumptions of \cref{tcouplingrate} hold. Consider
	$v^{\mathrm{ML}}\in L^2(D)$ defined by \cref{vML} as an approximation to the
	solution of $u$ of \cref{main}. By choosing the number of samples $M_\ell$ and
	number of levels $L$ as in \cref{ct}, we achieve
	$\norm{v^{\mathrm{ML}}-u}_{L^2 (\Omega,L^2(D) ) }\le \varepsilon$ with
	$\sorder{\epsilon^{-3+\beta/2}}$ cost, where $\beta=\min\Bp
		{\alpha,t\mu\alpha/(t+\mu\alpha)}$.%
\end{corollary}
\begin{proof}
	This is a consequence of \cref{ct} given values for
	$a,\beta,\gamma$. We have $a=2$ from \cref{mf} and $\gamma=2$ (due
	to the $2^d$-factor increase in vertices with each level). For the coupling rate $\beta$, we have  \cref{tcouplingrate}.
\end{proof}
This compares favourably to vanilla Monte Carlo, which has complexity
$\varepsilon^{-3}$, reducing the computational cost by a factor $\varepsilon^{\beta/2}$.

\begin{figure}
	\begin{center}
		\includegraphics[clip,trim=2cm 4cm 0cm 1cm]{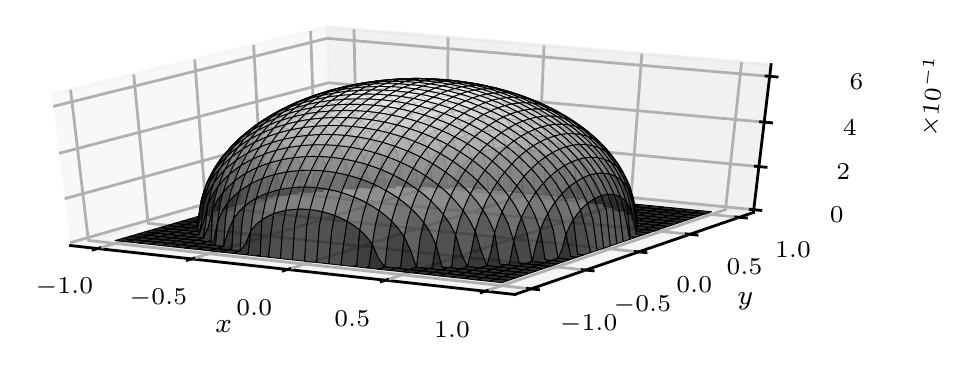}\\
		\includegraphics[clip,trim=2cm 4cm 0cm 1cm]{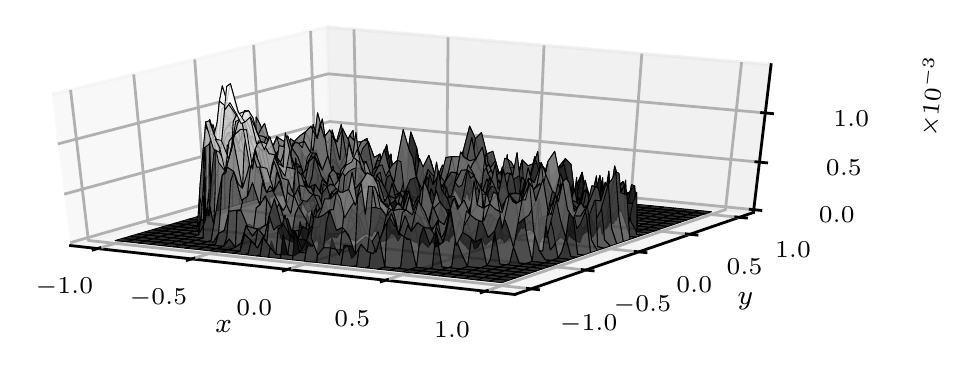}
	\end{center}
	\caption{The upper plot shows the solution $u(\vec x)$ defined
		\cref{ex1} for $\alpha=1$. The lower plot shows the absolute
		errors, for the computation with parameters  with
		$\varepsilon=1e-3$, $\ell_0=5$, $L=7$. On the finest level,
		$h=6.1\times 10^ {-5}$.}%
	\label{fig_ex1}
\end{figure}
\subsection{Numerical experiments}
We perform numerical experiments for $D=B(\vec 0,1)$. Define the
triangulation $\Bp{\tau_1^k}$ of the square $[-1,1]^2$ consisting of
the four triangles given by drawing diagonals. Let $\Bp{\tau_\ell^k}$
be the level-$\ell$ triangulation given by recursively dividing
triangles into four (using midpoints of edges). We apply WOS to the
vertices within the unit ball $D=B(\vec 0,1)$ and define
$v^{\mathrm{ML}}$ via  multilevel Monte Carlo (see \cref{vML}), using
piecewise-linear interpolants of sample averages to define $v_\ell$
on the triangulation $\Bp{\tau_\ell^k}$. The parameters for this
algorithm are the coursest level $\ell_0$, the finest level $L$, and
the tolerance $\varepsilon$. As test cases, we recall  two examples
on the unit ball where exact solutions are known
(e.g.,~\cite{Bucur2015-pp}).
\begin{example}
	\label[example]{ex1}%
	The problem with constant right-hand side and zero exterior condition on the unit ball,
	\[
		(-\Delta u)^{\alpha/2}u%
		=1\quad\text{on $D=B(\vec 0,1)$},\qquad %
		\text{$u=0$\quad on $D^{\coarse}$},
	\]
	has exact solution
	\[
		u(\vec x)= \frac{\Gamma(1-\alpha/2)}{2^\alpha%
			\Gamma(1+\alpha/2)%
			(\alpha/2) B(\alpha/2,1-\alpha/2)}%
		(1-\norm{\vec x}^2)^{\alpha/2},\qquad \vec x\in B(\vec 0,1).%
	\]
	The solution gives the mean first-exit time for an $\alpha$-stable
	Levy
	process from $B(\vec 0,1)$. It is plotted in \cref{fig_ex1} for
	$\alpha=1$, along with the error from the WOS approximation with
	$\ell_0=5$, $L=7$, and $\varepsilon=1e-3$.
\end{example}

\begin{figure}
	\begin{center}
		\includegraphics[clip,trim=2cm 4cm 0cm 1cm]{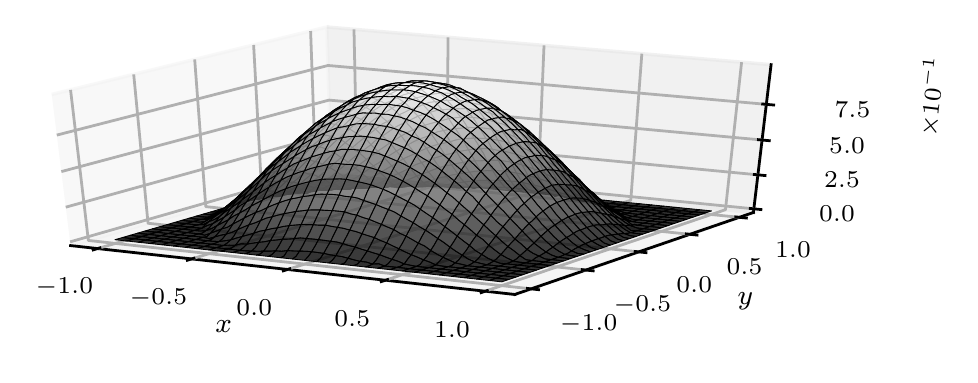}\\
		\includegraphics[clip,trim=2cm 4cm 0cm 1cm]{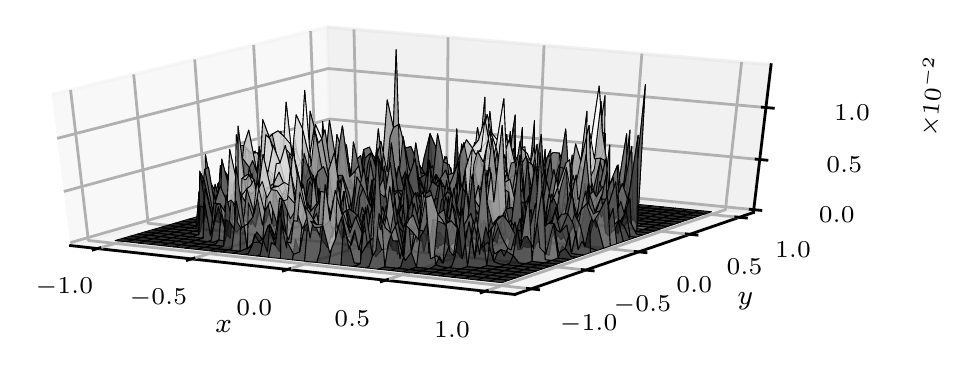}
	\end{center}
	\caption{The upper plot shows the solution $u(\vec x)$ of the
		problem in \cref{ex2} for $\alpha=1$. The lower plot shows the
		absolute errors, for the computation with parameters  with $
			\mathrm{tol}=1e-2$, $\ell_0=5$, and $L=7$. On the finest level,
		$h=6.1\times 10^{-5}$.}\label{fig_ex2}
\end{figure}
\begin{example}
	\label[example]{ex2}  The problem
	\[
		(-\Delta u)^{\alpha/2}u=f\quad\text{on $D=B(\vec 0,1)$},\qquad \text{$u=0$\quad
			on $D^{\coarse}$},
	\]
	for $f(\vec x)=(1-(1+\alpha/2)\,\norm{\vec x}^2),\ 2^\alpha\, \Gamma(2+\alpha/2)\,\Gamma(1+\alpha/2)$
	has exact solution
	\begin{equation}
		u(\vec x)%
		= (1-\norm{\vec x}^2)^{1+\alpha/2}.
		\label{exact2}%
	\end{equation}
	For $\alpha=1$, the solution is plotted in \cref{fig_ex2} along with
	the error for $\varepsilon=10^{-2}$, $\ell_0=5$, and $L=7$.
\end{example}
The third example  has less regular
coefficients and there is no explicit exact solution.
\begin{example}
	\label[example]{ex3} Let
	\[
		g(\vec x)=\sin(\norm{\vec x}^2),\qquad f(\vec x)=2+\norm{\vec x}^2.%
	\]
	The numerical solution of \cref{main} for $\alpha=1$ with $D=B(\vec 0,1)$
	is
	shown in \cref{fig_ex3}.
\end{example}
\begin{figure}
	\begin{center}
		\includegraphics[clip,trim=2cm 4cm 0cm 1cm]{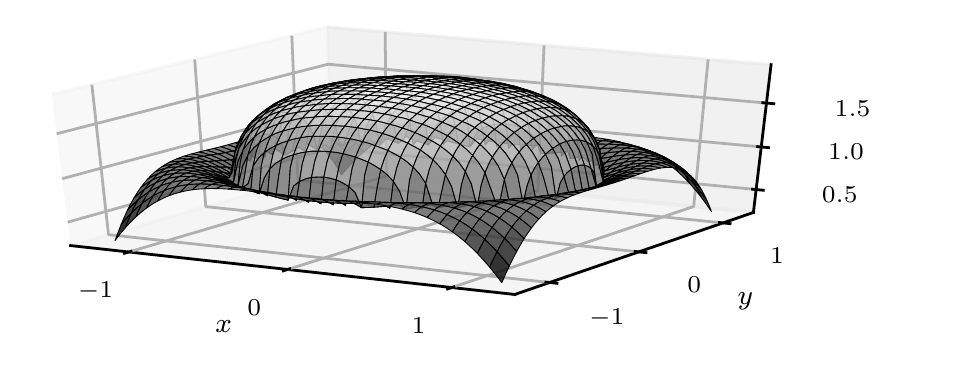}\\
	\end{center}
	\caption{The solution $u(\vec x)$ from \cref{ex3} for $\alpha=1$. }%
	\label{fig_ex3}
\end{figure}
\paragraph{Variance decay rates} We compute the variance decay rates for the
coupling numerically, and calculate the variance at level
$\ell$ defined by
\[
	V_\ell =\mean{\norm{ v^\fine_{\ell+1}-v^\coarse_\ell-\mean{v^\fine_{\ell+1}}-\mean{v^\coarse_\ell}}_{L^2(D)}^2}.%
\]
\cref{var_ex1} shows the variance $V_\ell$ against the mesh-width $h_\ell$ and
indicates decay rates $\beta\approx 1$ for \cref{ex1} independent of $\alpha$,
and rates $\beta=[0.25,0.47,0.67]$ for \cref{ex2} for $\alpha=
	[0.5,1.0,1.5]$.
Assuming $t=\mu=1$, the theoretical rate is $\beta\approx \alpha
	/(1+\alpha)$, which yields $\beta\approx[0.33, 0.5,0.6]$. The rate is
a
good prediction in \cref{ex3}. For \cref{ex1}, the right-hand side is
constant and the exterior condition is zero, which is a much simpler
scenario. Due to the constant right-hand side, zero
coupling-error results on the interior and improved performance can be expected, as found with the variance coupling rates of $\beta\approx 1$.

As a demonstration of the relative efficiency of multilevel Monte
Carlo and vanilla Monte Carlo methods, \cref{tol_cpu1} shows a plot
of CPU time against tolerance. Multilevel
Monte
Carlo is clearly more efficient.

We can also compare the computed solution to the exact $u\in L^2(D)$
by computing an approximate $L^2(D)$ norm of the error.
\cref{tab_residual} shows  errors for \cref{ex1,ex2} for seven values
of $\alpha$ and the computational time for a tolerance $10^{-2}$.
Good accuracy results, especially for larger values of $\alpha$, and
all results are computed in less than a minute  on a quad-core 3.2Ghz
i5-6500 CPU  with  8GB RAM (three cores are used in parallel for the
WOS samples). Small values of $\alpha$ give poorer results. This is
to be expected, due the sharp gradient near the border (see
\cref{exact2}), which is explained by a larger constant $C(\alpha,D)$
in \cref{3.3}.

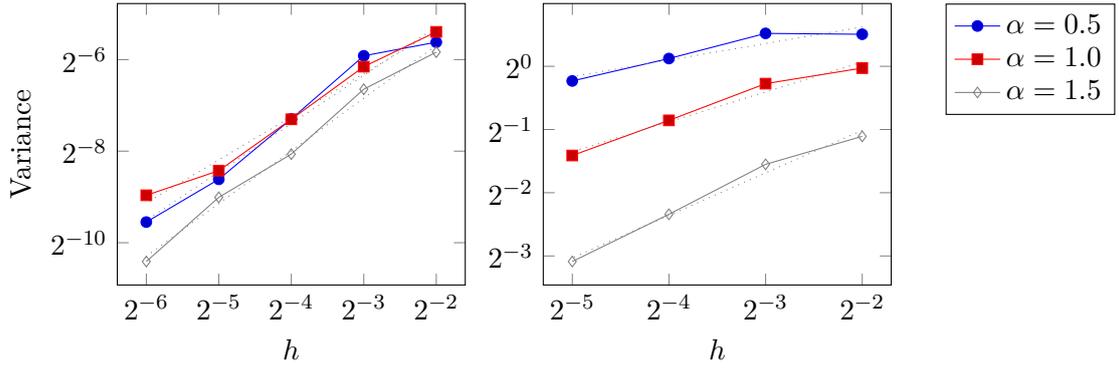
\begin{figure}
	\begin{center}
		\begin{tikzpicture}
			\begin{axis}[xlabel={$h$}, ylabel={Variance}, xmode={log},ymode=log,
					width=\textwidth*0.4,  log basis x={2},log basis y={2},legend style={at={(1.3,1.)},anchor=north}]
				;
				\addplot+[]
				coordinates {
						(0.25,0.020345886323608012) (0.125,0.01656016364257618) (0.0625,0.006365577998327078) (0.03125,0.0025527192310393363) (0.015625,0.001336149579428378)
					}; 

				\addplot+[color=gray,dotted,mark=none,forget plot]
				coordinates {
						(0.25,0.02560403255722408) (0.125,0.012319069531946171) (0.0625,0.005927170799902226) (0.03125,0.0028517862976672006) (0.015625,0.0013721023675741816)
					};
				\addplot+[]
				coordinates {
						(0.25,0.023784470476535785) (0.125,0.014045854391164417) (0.0625,0.006333880417721016) (0.03125,0.0029100794910107627) (0.015625,0.002004305048256096)
					}; 

				\addplot+[color=gray,dotted,mark=none,forget plot]
				coordinates {
						(0.25,0.024251386036990643) (0.125,0.01263299431019557) (0.0625,0.006580759755257166) (0.03125,0.0034280391404484015) (0.015625,0.0017857288193902446)
					};

				\addplot+[mark=diamond,color=gray]
				coordinates {
						(0.25,0.017502708501995916) (0.125,0.009985441845355924) (0.0625,0.003737019667417676) (0.03125,0.001943164772607361) (0.015625,0.0007353295920439526)
					}; 

				\addplot+[color=gray,dotted,mark=none,forget plot]
				coordinates {
						(0.25,0.019356139284517706) (0.125,0.008717828667681478) (0.0625,0.003926430553216735) (0.03125,0.0017684285246837748) (0.015625,0.0007964840851070584)
					};
			\end{axis}
		\end{tikzpicture}
		\begin{tikzpicture}
			\begin{axis}[xlabel={$h$},  xmode={log},ymode=log,
					width=\textwidth*0.4,  log basis x={2},log basis y={2},legend style={at={(1.4,1.)},anchor=north}]%
				;
				\addplot+[]
				coordinates {
						(0.25,1.420545558724419) (0.125,1.4339995367572) (0.0625,1.0887117634185621) (0.03125,0.8514557377945302)
					}; 

				\addlegendentry{$\alpha=0.5$}
				\addplot+[color=gray,dotted,mark=none,forget plot]
				coordinates {
						(0.25,1.5381443598745328) (0.125,1.283351317269097) (0.0625,1.0707646476502841) (0.03125,0.8933928809902237)
					};
				\addplot+[]
				coordinates {
						(0.25,0.9804253484000355) (0.125,0.8266789422373255) (0.0625,0.5526421818227631) (0.03125,0.37601918479580343)
					}; 

				\addlegendentry{$\alpha=1.0$}

				\addplot+[color=gray,dotted,mark=none,forget plot]
				coordinates {
						(0.25,1.0474355847987955) (0.125,0.7547040515409867) (0.0625,0.5437835163121671) (0.03125,0.3918098915847228)
					};

				\addplot+[mark=diamond,color=gray]
				coordinates {
						(0.25,0.46384430135807486) (0.125,0.34130788834672887) (0.0625,0.1974426519966588) (0.03125,0.11761033894003425)
					}; 

				\addlegendentry{$\alpha=1.5$}

				\addplot+[color=gray,dotted,mark=none,forget plot]
				coordinates {
						(0.25,0.49564883001380167) (0.125,0.31090344420200505) (0.0625,0.19501902509076374) (0.03125,0.12232871927478844)
					};

			\end{axis}
		\end{tikzpicture}
	\end{center}
	\caption{Plots of coupling variance $V_\ell$ against mesh-width $h_\ell$ for
		the		multilevel Monte Carlo, with dotted lines showing best linear fits.
		The left-hand plot shows \cref{ex1} with slopes $1.05, 0.94,1.15$; the
		right-hand side shows \cref{ex3} with slopes for  $0.25,0.47,0.67$
		for $\alpha=0.5,1.0,1.5$ respectively.}\label{var_ex1}
\end{figure}
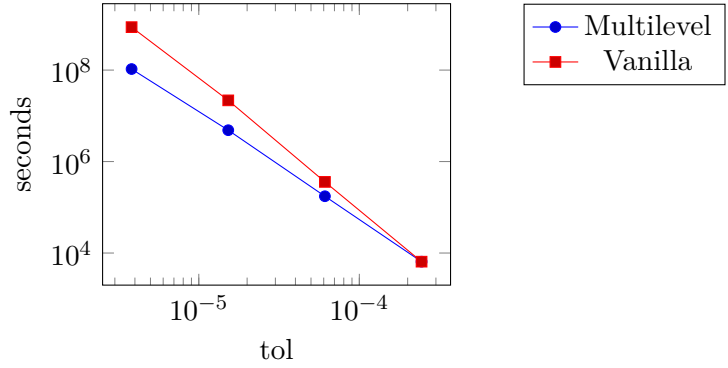
\begin{figure}
	\begin{center}
		\begin{tikzpicture}
			\begin{axis}[xlabel={tol}, ylabel={seconds}, xmode={log},ymode=log,  width=\textwidth*0.4,  log basis x={10},log basis y={10},legend style={at={(1.5,1.)},anchor=north}];
				\addplot+[]
				coordinates {
						(0.000244140625,6474.757338701406) (6.103515625e-5,174532.96761901176) (1.52587890625e-5,4.8501413236256e6) (3.814697265625e-6,1.0540719736423744e8)
					};
				\addlegendentry{Multilevel}
				\addplot+[]
				coordinates {
						(0.000244140625,6474.757338701406) (6.103515625e-5,359285.61420145415) (1.52587890625e-5,2.178288723919411e7) (3.814697265625e-6,8.688204845658375e8)
					};

				\addlegendentry{Vanilla}
			\end{axis}
		\end{tikzpicture}
	\end{center}
	\caption{CPU execution time against tolerance $\varepsilon$ for the
	multilevel
	Monte Carlo and vanilla (i.e., $\ell_0=L$) Monte Carlo methods, based
	on \cref{ex2}.
	The multilevel parameter $\ell_0=6$ and $\varepsilon=2^{-2L}$, where
	$L$ is varied.}\label[figure]{tol_cpu1}
\end{figure}



\begin{table}
	\begin{center}
		\begin{tabular}{c  l  l r }
			\multicolumn{2}{c}{\cref{ex1}}            \\
			$\alpha$ & rel error & abs err & cpu time \\
			\midrule
			0.1      & 0.079     & 0.13    & 1.9      \\
			0.2      & 0.064     & 0.10    & 2.8      \\
			0.5      & 0.034     & 0.051   & 3.6      \\
			1.0      & 0.014     & 0.018   & 5.9      \\
			1.5      & 0.008     & 0.0090  & 17.0     \\
			1.8      & 0.00081   & 0.0085  & 36.6     \\
			1.9      & 0.0062    & 0.0065  & 60.0
		\end{tabular}
		\qquad
		\begin{tabular}{c  l  l r}
			\multicolumn{2}{c}{\cref{ex2}}            \\
			$\alpha$ & rel error & abs err & cpu time \\
			\midrule
			0.1      & 0.0093    & 0.0093  & 17.9     \\
			0.2      & 0.0057    & 0.0056  & 1.60     \\
			0.5      & 0.0057    & 0.0054  & 2.54     \\
			1.0      & 0.0058    & 0.0052  & 8.73     \\
			1.5      & 0.0090    & 0.0075  & 17.79    \\
			1.8      & 0.0084    & 0.0068  & 40.47    \\
			1.9      & 0.0085    & 0.0068  & 68.37
		\end{tabular}
	\end{center}
	\caption{Computed $L^2(D)$ relative and absolute errors and CPU execution time in seconds for a range of $\alpha$.}\label[table]{tab_residual}
\end{table}

\section{Leading eigenvalue using the Arnoldi algorithm}\label{eigenvalue}
The Arnoldi algorithm is a well-known iterative method for computing
the leading eigenvalues of a large sparse matrix, based on projecting
the matrix onto a Krylov subspace.  See~\cite{Arnoldi_1951,Trefethen1997-sw,Craig2003-oa,Saad2011-bx}.
We show how to use Arnoldi to compute the smallest eigenvalue of
the fractional
Laplacian. That is, we seek the smallest $\lambda>0$ such that
\[
	(-\Delta)^{\alpha/2}w=\lambda\, w \text{ on $D$},\qquad
	\text{$w=0$ on $D^\coarse$}
\]
for some non-trivial function $w\colon\real^d\to \real$.
The first
step is to discretise the fractional
Laplacian and express the problem for a finite-dimensional linear operator.
To do this, we consider a triangular mesh $\mathcal{T}$ with vertices $\vec
	z^1,\dots,\vec z^N$. We consider $\vec v\in\real^N$ corresponding to values
of a function at vertices of the mesh and the solution operator for \cref{main}
with $g=0$ and $f=I_h \vec v$, the piecewise-linear interpolant for $\vec
	v$ on the mesh $\mathcal{T}$.
That is,
for $\vec v\in \real^N$, let $u\colon
	D\to \real$  be the solution to $(-\Delta )^{\alpha/2} u= f$ on $D$, where
$f=I_h \vec v$, and $u=0$ on $D^\coarse$.
Then, we denote by $A^{-1}$ the linear mapping from $\vec v\in \real^N$
to ${\vec u}=[u(\vec z^1),\dots,v(\vec z^N)]\in\real^N$.
By applying the Arnoldi algorithm to $A^{-1}$, we find the largest eigenvalue
of $A^{-1}$ and hence the smallest eigenvalue of $A$, which is the fractional
Laplacian approximated on $\mathcal{T}$. We use this as our
approximation to the smallest eigenvalue of the fractional Laplacian.

In practice, evaluating $A^{-1}\vec v=\vec u$ exactly is impossible and
we will be using the WOS algorithm. This means we will be using the Arnoldi
algorithm with inexact solves and exploiting the theory for variable-accuracy
Arnoldi algorithms started by  \cite{Bouras2000-bg,Simoncini2005-ra} and
developed
further in \cite{Berns-Muller2006-na,Freitag2010-kz}.  It turns
out that the accuracy of
the solves can be reduced as the Arnoldi algorithm proceeds, without loosing
accuracy on the computed eigenvalue. This leads to significant speed ups.
We develop the appropriate variable accuracy criterion for the WOS solve,
by establishing a criterion on the variance of the WOS solution necessary
for a certain confidence interval in the computed eigenvalue.

We now describe the algorithm. Throughout, $\norm{\cdot}$ denotes the
Euclidean norm.

\paragraph{Inexact Arnoldi iteration} To determine the smallest eigenvalue
of the fractional Laplacian $(-\Delta)^{\alpha/2}$ on a domain $D$, choose
a triangulation $\mathcal{T}$ of the domain with $N$ vertices and  consider
$N$-vectors of values at the vertices.
\begin{algorithm}
	\label[algorithm]{alg}
	\begin{enumerate}
		\item Choose initial unit-vector $\vec v_1\in\real^N$. Set $k=1$.
		\item Evaluate  $A^{-1}\vec v_k=\vec u+\vec f_k$, where $\vec f_k$ is
		      the error resulting from the WOS solve.
		\item Gram--Schmidt: for $i=1,\dots,k$, let $h_{ik}\coloneq \vec
			      v_i^\trans \vec u$ and compute $\tilde{\vec u}\coloneq \vec
			      u-\sum_{i=1}^{k} h_{ik}
			      \vec
			      v_i$, to produce $\tilde{\vec u}\in\real^N$ orthogonal to the current
		      Krylov space, given by $\text{span}\{\vec v_1,\dots,\vec v_k\}$, and
		      coefficients $h_{ik}$ for $i=1,\dots,k$. Let $h_{k+1,k}\coloneq
			      \norm{\tilde{\vec u}}$ and add the vector $\vec v_
				      {k+1}\coloneq \tilde{\vec u}/h_{k+1,k}$ to the Krylov space. Let
		      $V_
				      {k}$ be
		      the matrix with columns $\vec v_1,\dots, \vec v_{k}$.
		\item Compute eigenvectors $\vec w_j$ and eigenvalues $\theta_j$
		      (known as Ritz values) for the leading $k\times k$ submatrix $H_k$ of
		      the upper Hessenberg matrix $(h_{ij})$. The largest eigenpair $
			      (\vec
			      w, \theta)$ defines an approximate eigenvector for $A^{-1}$ by $V_k \vec w$ and approximate eigenvalue $\theta^{-1}$.
		\item  Increase $k$ and repeat.
	\end{enumerate}
\end{algorithm}
For finite-dimensional problems with exact solves ($\vec f_k=\vec 0$ for
all $k$), the algorithm is
expected to converge as $k\to \infty$ to the leading eigenpair of
$A^{-1}$. In our case, the algorithm introduces errors at several
stages: First, we represent the WOS solutions on a triangular mesh and
the WOS
algorithm must  evaluate the right-hand side function everywhere on
the domain $D$. We use a piecewise-linear interpolant and this leads to
an approximation error.
Additionally, there is a Monte Carlo error  on $\vec u$ due to the finite
number of samples. We assume the error due to linear interpolation is
negligible compared to Monte Carlo error; as the size of this error
is quadratic in the mesh width $h_\ell$, this can be achieved by
choosing the triangulation fine enough. For the Monte Carlo error, we develop
a theory for the resulting  error in the computed eigenvalue and  a practical
criteria for the tolerance  for the WOS solve.
\subsection{Choosing the WOS tolerance}
To analyse the error due to the WOS solve at step $k$, we write
\[
	A^{-1} \vec v_{k}%
	=\vec u+\vec f_k%
	= \sum_{i=1}^k h_{ik}\vec v_k
	+h_{k+1,k}\vec v_{k+1}+\vec f_k,
\]
where $\vec f_k$ represents the error due to the $k$th WOS solve. The right-hand
side is given by the representation of $\vec u$ in the Krylov subspace (using the entries in the Hessenberg matrix $H_k$).
We determine a criterion for relating the accuracy in the WOS solve
(the size of ${\vec f_k}$) in terms of the desired eigenvalue
accuracy. Stacking the expressions above, we have the well-known
Arnoldi relation
\[
	A^{-1} V_m = V_m H_m + \vec v_{m+1} h_{m+1} \vec e_m^\trans+F_m,
\]
where $\vec e_m$ is the $m$th standard basis vector in $\real^N$ and
$F_m$ is the matrix of column vectors $[\vec f_1,\dots,\vec f_m]$.

In the case of zero error $\vec f_k=\vec 0$ for $k=1,\dots,m$, the
size of the
eigenvalue residual $\vec r_m=A^{-1}  V_m \vec w-\theta^{-1} V_m \vec
	w$ is
given by $\norm{\vec r_m}=h_{m+1,m} \abs{\vec w^\trans \vec e_m}$,
which follows easily  from the Arnoldi relation
\[
	A^{-1} V_m = V_m H_m + \vec v_{m+1} h_{m+1} \vec e_m^\trans
\]
by taking the inner product with the eigenvector $\vec w$ of $H_m$.
For the case of inexact solves, the residual itself is not readily available
and the quantity $\norm{\vec r_m}\coloneq h_{m+1,m} \abs{ \vec w^\trans
		\vec e_m}$
serves as a computationally
convenient proxy.

For the next theorem, we will need the following non-degeneracy assumption~\cite[(3.3)--(3.4)]{Simoncini2005-ra}. The notation $\sigma_{\min}(A)$ refers to the smallest singular value of $A$.

\begin{assumption}
	\label[assumption]{ass:e}%
	For $k\le m$, there exists an eigenpair $(\theta^{(k-1)}, \vec w^{(k-1)})$
	of $H_{k-1}$ sufficiently close to an eigenpair $(\theta, \vec w)$ of $H_m$
	in the sense that
	\[
		\norm{\vec r_{k-1}}%
		\le \delta^2_{m,k-1} \frac{1}{4\norm{\vec s_m}},\qquad%
		\delta_{m,k}\coloneq \sigma_{\min}(H_m-\theta^{(k)}I),
	\]
	where $\vec s_m^\trans\coloneq [(\vec w^{(k-1)})^\trans,\vec 0^\trans]
		H_m - \theta^{(k-1)} [(\vec w^{(k-1)})^\trans,\vec 0^\trans]\in\real^m$,
	and
	\[
		\text{$\abs{\theta^{(k-1)}-\theta_j} > 2 \frac{\norm{\vec s_m}\,\norm{\vec r_{k-1}}}{\delta_{m,k-1}}$,\; for all eigenvalues $\theta_j\ne \theta$ of $H_m$. }
	\]
\end{assumption}

In the following theorem, we develop the accuracy criterion for the WOS
solves. We use the  spectral gap $\delta^{(k-1)}$ for $H_
		{k-1}$: let $\Lambda(H_{k})$ denote the set of eigenvalues of $H_k$ and
\begin{equation}
	\delta^{(k-1)}
	\coloneq \min_{\theta\in \Lambda(H_{k-1})- \theta^{(k-1)}}%
	\abs{\theta^{(k-1)}-\theta},\label{gap}
\end{equation}
where $\theta^{(k-1)}$ is the leading eigenvalue of $H_{k-1}$. The
quantity depends on the spectrum of $H_{k-1}$, which is easily computable
as $k$ is generally small. We also use the $\sigma$-algebra $\mathcal{F}_m$
generated by the random variables used in the WOS solves up to step $m$, so that the residual $\vec r_m$ is $\mathcal{F}_m$-measurable. We will use the conditional expectation $\mean{\cdot\,|\,\mathcal{F}_k}$ to average over input random variables (WOS samples from the field 	solves) for $m>k$.
\begin{theorem}\label{tt}
	Let $(\theta, \vec w)$ be an eigenpair of the $m$th Hessenberg matrix $H_m$
	of an inexact Arnoldi iteration as described in \cref{alg} such that
	\begin{equation}
		\norm{\vec r_m}%
		\coloneq h_{m+1,m}\abs{\vec e_m^\trans \vec w}%
		\le \mathrm{tol}.
		\label[ineq]{nr}
	\end{equation}
	Suppose that the eigenvalue $\theta$ is simple and the eigenvector is normalised,
	$\norm{\vec w}=1$. Fix $\varepsilon>0$ and suppose that the WOS error vector $\vec f_k$ at step $k$ satisfies
	\[
		\pp{\mean{\norm{{\vec f}_k}^2\,| \mathcal{F}_{k-1}}}^{1/2}%
		\le\frac{\varepsilon}{m}\times
		\begin{cases}{\displaystyle \frac{\delta^{(k-1)}}{\norm{\vec r_{k-1}}} }, &
			\text{if $k>1$ and \cref{ass:e} holds,}                                          \\[1.2em]
			1 ,                                                          & \text{otherwise.}
		\end{cases}
	\]

	Then, $V_m \vec w$ is an approximate eigenvector of $A$ with eigenvalue
	$\theta^{-1}$, in the sense that \begin{equation}
		\prob{    \norm{A^{-1} V_m \vec w- \theta^{-1} V_m \vec w}>2\,\mathrm{tol}}%
		\le \frac{\varepsilon^2}{\mathrm{tol}^2}%
		\,\mean{\frac 1m\norm{\vec \alpha}^2},
		\label[ineq]{pt}
	\end{equation}
	where $\vec \alpha$ has entries $\alpha_k$ satisfying
	\begin{equation}
		\abs{\alpha_k}%
		\le
		2\,\frac{\delta^{(k-1)}}{\delta_{m,k-1}} \text{ if \cref{ass:e} holds,}\quad
		\alpha_k=1 \text{ otherwise.}
		\label[ineq]{cookie}
	\end{equation}
	In particular, there exists an eigenvalue $\mu$ of $A$
	such that
	\[
		\prob{ \abs{\mu^{-1}-\theta^{-1} }>2\,\mathrm{tol}}%
		\le \frac{\varepsilon^2}{\mathrm{tol}^2}%
		\mean{\frac 1m\norm{\vec \alpha}^2}.
	\]
\end{theorem}
\begin{proof}
	For inexact solves, the Arnoldi relationship is
	\[
		A^{-1} V_m - F_m%
		= V_m H_m + h_{m+1,m} \vec v_{m+1} \vec e_m^\trans,
	\]
	where the columns of $F_m$ are $\vec f_1,\dots,\vec f_m$. If $(\theta,
		\vec w)$ is an eigenpair of $H_m$, then
	\[
		A^{-1} V_m \vec w - F_m \vec w %
		= \theta\, V_m \vec w%
		+ h_{m+1,m} \vec v_{m+1} (\vec e_m^\trans \vec w).
	\]
	Hence, the eigenvalue residual
	\begin{equation}
		\norm{    A^{-1} V_m \vec w - \theta V_m \vec w}%
		= h_{m+1,m} \abs{ \vec e_m^\trans \vec w} + \norm{F_m \vec w}.\label{evres}
	\end{equation}
	The first residual corresponds to~\eqref{nr} and is monitored during
	the Arnoldi iteration. The second term, we refer to as the extra
	residual, is due to the inexact solves and we analyse that now. Following~\cite{Simoncini2005-ra}, it can be written as
	\begin{align*}
		{F_m \vec w}%
		 & =\sum_{k=1}^m \vec{f}_k (\vec e_k^\trans \vec w) %
		=\sum_{k=1}^m \frac{\vec{f}_k \norm{\vec r_{k-1}}}{\delta^{(k-1)}}
		(\vec e_k^\trans \vec w)\frac{\delta^{(k-1)}}{\norm{\vec r_{k-1}}},
	\end{align*}
	where $\delta^{(k-1)}$ is the spectral gap defined in \cref{gap}.
	%
	Write,
	\begin{align*}
		{F_m \vec w}%
		%
		=\sum_{k=1}^m \alpha_k \,\beta_{k-1}\, \vec f_k,
	\end{align*}
	for $\beta_{k-1}=\norm{\vec r_{k-1}}/\delta^{(k-1)}$ and $\alpha_{k}=
		\norm{\vec r_{k-1}}^{-1}  (\vec{e}_k^\trans \vec w )\,\delta^{(k-1)}$ if
	\cref{ass:e} holds, and $\beta_{k-1}=\alpha_k=1$ otherwise.
	Here the subscripts for $\alpha_k$ and $\beta_k$ indicate that they
	are $
		\mathcal{F}_k$-measurable. The Cauchy--Schwarz inequality provides that
	\begin{align*}
		\mean{\norm{ F_m \vec w}^2}%
		 & \le \mean{\sum_{k=1}^m m \,\beta_{k-1}^2 \norm{\vec f_k}^2} \, %
		\mean{ \frac 1m\sum_{k=1}^m \alpha_k^2}.
	\end{align*}
	By Chebyshev's inequality,
	\begin{align*}
		\prob{\norm{F_m \vec w}>
			\mathrm{tol}}
		 & \le
		\frac{1}{\mathrm{tol}^2}\norm{\sum_{k=1}^m {m}\,\mean{\beta_{k-1}^2\norm{\vec f_k}^2 } }\,
		\mean{\frac 1m\sum_{k=1}^m {\alpha}_k^2}.
	\end{align*}


	If \cref{ass:e} holds, the condition on the WOS error $\vec f_k$ implies that
	\[
		\mean{\norm{\vec f_k}^2\,| \,\mathcal{F}_{k-1}}%
		\le \frac{(\varepsilon\, \delta^{
				(k-1)})^2}{ (m\,\norm{\vec r_{k-1}})^2}%
		=\frac{\varepsilon^2}{m\,\beta_ {k-1}^2}.
	\]
	Otherwise, 	$\mean{\norm{\vec f_k}^2\,| \,\mathcal{F}_{k-1}}%
		\le \varepsilon^2/m^2.
	$ Hence,
	\[
		\prob{\norm{F_m \vec w}>\mathrm{tol}}%
		\le \frac{\varepsilon^2}{\mathrm{tol}^2  } \,\mean{\frac 1m\norm{\vec \alpha}^2}.
	\]
	Using \cref{evres}, this implies \cref{pt} under the condition \cref{nr}.

	Under \cref{ass:e},  \cite[Proposition~2.2]
	{Simoncini2005-ra} implies that
	\[
		\abs{\vec e_k^\trans \vec w}%
		\le 2 \frac{\norm{\vec r_{k-1}}}
		{\delta_{m,k-1}}.
	\]
	Here, it is important to note
	that $\delta_{m,k-1}$  is not $\mathcal{F}_{k}$-measurable in general and depends on the full Arnoldi run. Then,
	\[
		\abs{\alpha_k}%
		\le 2 \frac{ \delta^{(k-1)}}{ \delta_{m,k-1}},
	\]
	which is \cref{cookie}.
	The final statement is a consequence of the Bauer--Fike theorem for
	normal matrices~\cite{Golub2013-ip}, which says that
	the error in the eigenvalue is bounded by the eigenvalue residual.
\end{proof}

This theorem suggests a practical way of choosing the WOS tolerance
at step $k$ in dependence on a given eigenvalue-residual tolerance $\mathrm{tol}$, parameter $\varepsilon$, and computed residual $\norm{\vec r_k}$. For $B>1$, to achieve
\[
	\prob{\text{eigenvalue residual}>2\,\mathrm{tol}}%
	\le \frac1{B^2}, 
\]
we assume that $\mean{\frac1m \norm{\vec \alpha}^2}\approx 1$,
and
choose $\varepsilon= \mathrm{tol}/B$ for $B> 1$. Then,
\[
	\prob{\text{eigenvalue residual}>2\,\mathrm{tol}}%
	\le \frac{\varepsilon^2}{\mathrm{tol}^2} \mean{\frac1m \norm{\vec \alpha}^2}%
	\le \frac1{B^2}.
\]
For the $k$th WOS solve, from \cref{tt}, we demand that
\[
	\pp{\mean{\norm{\vec f_k}^2| \mathcal{F}_{k-1} } }^{1/2}%
	\le \frac{1}{B\, m}\,\mathrm{tol}\, \frac{\delta^{(k-1)} }{\norm{\vec
			r_{k-1}}}.
\]
This leads to a relaxed accuracy condition for the WOS calculation
if the computed eigenvalue residual $\norm{\vec r_{k-1}}$ is smaller than
the spectral gap $\delta^{(k-1)}$. It is simple to implement and requires
computing the spectrum of the $k\times  k$ Hessenberg matrix $H_k$ at each
step (to determine $\delta^{(k-1)}$) and monitoring the variance in the WOS Monte Carlo calculation.

\subsection{Leading eigenvalue on the unit ball}
Dyda~\cite{Dyda2012-kl} provides upper and lower bounds on the
leading eigenvalue for the fractional Laplacian on the unit ball,
gained by rigorous analytical methods. We compare this to the
eigenvalues computed by \cref{alg}. \cref{tab2} shows the results of
a computation with $\mathrm{tol}=0.01$, $B=3$, and WOS multilevel parameters
$\ell_0=3$, and $L=7$ with five Arnoldi iterations. The Arnoldi
iteration produces estimates that are very close to Dyda's upper
bound (the error relative to the upper bound are in the range $
	3\times 10^{-4}$ to $ 10^{-3}$). The computations take between take
fifteen and thirty minutes  for a Julia implementation on a
quad-core 3.2Ghz i5-6500 CPU  with  8GB RAM (three cores are used in
parallel for the WOS samples). The run times are compared in \cref{last}
to the Arnoldi algorithm without variable accuracy, taking the same tolerance
in both runs. The variable accuracy algorithm is twice as fast, for the
same level of accuracy.
\begin{table}
	\begin{tabular}{c  l l  l l r r}
		                            & \multicolumn{2}{l}{Dyda eigenvalue} &
		\multicolumn{3}{l}{Arnoldi} & error
		relative                                                                                                                               \\
		$\alpha$                    & lower bound                         & upper bound & eigenvalue & residual & seconds & to upper bound     \\
		\midrule
		0.1                         & 1.04874                             & 1.05096     & 1.05187    & 0.0672   & 885     & $8\times 10^{-4}$  \\
		0.2                         & 1.10549                             & 1.10993     & 1.11103    & 0.0896   & 933     & $9 \times 10^{-4}$ \\
		0.5                         & 1.3313                              & 1.34374     & 1.34464    & 0.8147   & 838     & $6\times 10^{-4}$  \\
		1.0                         & 1.96349                             & 2.00612     & 2.00689    & 0.095    & 907     & $3\times 10^{-4}$  \\
		1.5                         & 3.13569                             & 3.27594     & 3.27789    & 0.017    & 1008    & $5\times 10^{-4}$  \\
		1.8                         & 4.28394                             & 4.56719     & 4.57029    & 0.014    & 1268    & $6 \times 10^{-4}$ \\
		1.9                         & 4.77496                             & 5.13213     & 5.13936    & 0.00075  & 1711    & $1\times 10^{-3}$
	\end{tabular}
	\caption{\cite{Dyda2012-kl} provides the lower and upper bounds on
		the smallest eigenvalue of the fractional Laplacian on $B(\vec 0,1)$ shown. These are compared to the result of an Arnoldi computation with five iterations and the variable accuracy methods described.}\label{tab2}
\end{table}

\begin{figure}
	\begin{center}
		\begin{tikzpicture}
			\begin{axis}[xlabel={$\alpha$}, ylabel={Elapsed time (seconds)}, width=\textwidth*0.4];
				\addplot+[]
				coordinates {
						(0.5,349) (1.0,320) (1.5,550)};
				%
				\addplot+[style=dashed]
				coordinates {						(0.5,922) (1.0,1228) (1.5,1797)};%
				%
			\end{axis}
		\end{tikzpicture}
		\qquad
		\begin{tikzpicture}
			\begin{axis}[xlabel={$\alpha$}, ylabel={Relative error},
					width=\textwidth*0.4
				];
				\addplot+[]
				coordinates {						(0.5,0.070) (1.0,0.056) (1.5,0.018)};
				\addplot+[style=dashed]
				coordinates {						(0.5,0.107) (1.0,0.0700) (1.5,0.02)};%
			\end{axis}
		\end{tikzpicture}
	\end{center}
	\caption{Comparison of run times with (solid) and without (dashed) variable
		accuracy, using $\mathrm{tol}=1e-2$, $L=7$, $\ell_0=5$, $B=3$. The
		left-hand plot shows run times in seconds against $\alpha$. The right-hand
		plot show the relative accuracy (relative to the upper bound of Dyda's
		interval). }\label{last}
\end{figure}
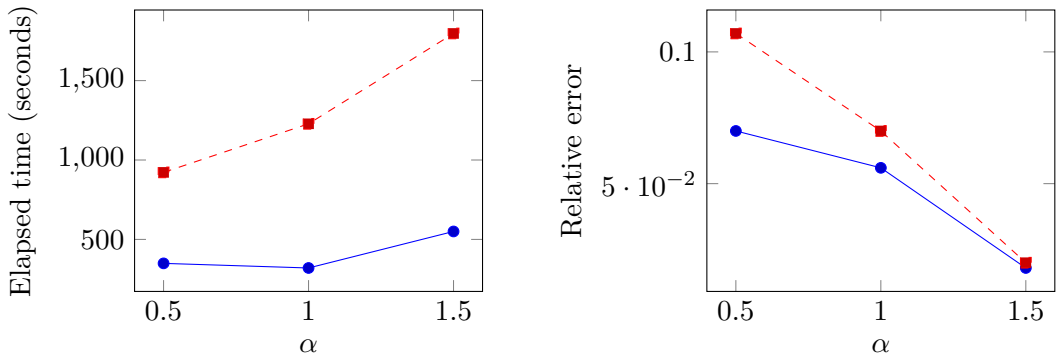

\section{Conclusion}%
\label{conclusion}%
We have discussed Walk Outside Spheres for simulating
the whole field rather than a point value of the solution $u\colon D
	\to \real$ of \cref{main}, extending the algorithm of \cite{kos}. By
using the multilevel Monte Carlo algorithm, we improved substantially
on a naive method based on independent sampling at vertices. The
improvement is demonstrated analytically (an improvement in the
complexity for accuracy $\varepsilon$ of factor $\varepsilon^\beta$
where $\beta\ge \min\Bp{\alpha,t\mu\alpha/ (t+\mu\alpha)}$. The parameters
$\mu,t$ lie in the range $(0,1)$ and are generally unknown; numerical
examination of the relevant
assumptions shows that $t$ can
be chosen close to one and $\mu$ depends on $\alpha$ ($\mu$ may be chosen
close to
one for small $\alpha$ and must be reduced substantially as $\alpha\to 2$).
The improvement is also
demonstrated numerically, by looking at two problems with exact
solutions and a third problem where variance estimates were made.
Numerical
experiments show the complexity bounds are pessimistic, even for $t=\mu=1$.
This
is because the assumptions and analysis are based on a pair of coupled WOS processes,
while
the $L^2(D)$ error depends on an average over a WOS path for each initial
vertex.

There are several deterministic approaches to the numerical approximation of the exterior-value problem for the fractional Laplacian \cite{Lischke2018-qb}. We compare our results to the complexity and error analysis for the adaptive finite-element method (AFEM) of \cite{Ainsworth2017-xf,Ainsworth2018-oo}. By using  a posterioi error estimates and  sparse approximations to the  dense linear systems resulting from the global coupling in the fractional Laplacian,  it converges in the $L^2(D)$ sense in two dimensions with $\sorder{n^{-1/2-\alpha/4}}$, where $n$ is the number of degrees of freedom.  Due to the use of a clustering technique in assembly of the linear system, the solution can be computed in $\order{n \log^4(n)}$ operations. In terms of an $L^2(D)$ accuracy of $\epsilon$, this method requires $\sorder{\epsilon^{-2+2\alpha/(2+\alpha)-\delta}}$  operations on a polygonal domain in two dimensions (for any $\delta>0$). In contrast, the WOS field method on a uniform mesh takes $\sorder{\epsilon^{-3+\beta/2}}$ operations to achieve a root-mean-square $L^2(D)$ accuracy of $\epsilon$,    where $\beta$ is given in \cref{cor:complexity}.

Though AFEM is one order of magnitude faster, the field WOS method has significant potential. First, the WOS method  is trivial to parallelise and this  means the constant associated to the  complexity analysis can be made small, given sufficient parallel resources. This can be very significant in practical situations.   Second, the WOS method developed here is a classical Monte Carlo method, depending on a sequence of independent samples from certain probability distributions. The complexity of such methods depends on the $1/\sqrt{M}$ sampling error from  $M$ independent samples, even for multilevel Monte Carlo methods. Very often, the complexity  is substantially improved by employing quasi-Monte Carlo techniques. In situations where the sample  depends on an infinite number of random variables and the importance of these random variables decays suitably rapidly, the $\sorder{1/\sqrt{M}}$ error can be replaced by $\sorder{1/M^{1-\delta}}$, for $\delta>0$ \cite{Hickernell2000-eo, Graham2011-yz}. This sort of analysis has not been completed for the field WOS method and is beyond the scope of the present paper. It is worth noting that each $L^2(D)$ sample depends on a finite but unbounded number of random variables (depending on the number of steps for the WOS path to exit the domain), but only one set of random variables is used for each vertex in the mesh (due to coupling) and the importance of these random variables decays geometrically (as the exit time is geometrically distributed \cite{kos}). It will be a subject of future research to develop a quasi-Monte Carlo-based field-WOS method with  improved complexity. If the $1/\sqrt{M}$ is replaced by $1/M$ in the complexity analysis, the overall complexity  improves by a factor $\epsilon$. There is a potential also to improve the method by using an adaptive mesh. At this point, the method becomes  competitive with AFEM. If these issues can be overcome,  a larger class of particle methods for solving PDEs can be coupled in the same way to give efficient field solvers.

Finally, we used the WOS algorithm to compute the leading eigenvalue
of the fractional Laplacian. We developed a criterion for accuracy
at the $k$th Arnoldi iteration based on the spectral gap of the
Hessenberg matrix and the residual. The method is shown to give
accurate results by comparing to analytical results of
\cite{Dyda2012-kl}. This algorithm can be developed
further to get more leading eigenvalues and to incorporate the
implicitly restarted Arnoldi method. Note the shift strategy (i.e., solving for $(A-sI)\vec x=\vec b$ for
a shift
value $s)$ commonly used in eigenvalue solvers is not easy to apply with
WOS; the implicitly restarted Arnoldi algorithm allows shifts to be introduced
implicitly and only solves for the fractional Laplacian are required. Here again variable accuracy
strategies
are available \cite{Freitag2010-kz} and could be adapted to the randomly inexact inner
solves. This is the subject of future work.
%
\appendix

\section{On \cref{hole,skin2}}%
\label{app}%
We verify that \cref{hole,skin2} are
realistic, by computing the relevant quantities numerically
for a square domain. For some
$t,\lambda,\mu\in (0,1)$ and $A>0$, we wish to establish that
\[
	I_1(\vec x_0)\coloneq \mean{\frac{\Phi(\vec x_1)}{\Phi(\vec x_0)}\,1_{\Bp{\vec
		x_1\in D}}\,|\,\vec x_0} \le \lambda,\qquad \Phi(\vec x)\coloneq \max\Bp{A,\frac{1}{d(\vec x)^t}}%
\]
and
\[
	I_2(\vec x_0,\vec y_0)\coloneq  \mean{\pp{\frac{\norm{\vec x_1-\vec y_1}}{\norm{\vec x_0-\vec
				y_0}}}^{\mu\alpha}\,1_{\Bp{\vec x_1,\vec y_1\in D}}} \le \lambda,\qquad \text{for all $\vec x_0,\vec y_0\in D$}.
\]
If we replace $A$ by $A/\diameter^t$ for the diameter $\bar d$ of $D$, it
is clear both conditions are invariant to rescaling the domain. Hence, we
focus on a box $D=[0,1]\times [0,1]$. We compute the values by a simple
Monte Carlo method using $\vec x_1=\vec x_0+\vec \Theta\,d(\vec x_0)/\sqrt{\beta}$,
for $\vec\Theta\sim \mathrm{U}(S^1)$ and $\beta\sim \mathrm{Beta}(\alpha/2,1-\alpha/2)$.

The expression for $I_2$  involves one parameter $\mu$ but must be checked
for every pair of $\vec x_0,\vec y_0\in D$. We draw twenty
$\vec x^j_0,\vec y^j_0$
independently from $\mathrm{U}(D)$ and evaluate $I_2(\vec x_0^j,\vec y_0^j)$
using $10^6$ samples of $\vec x_1,\vec y_1$ and record the maximum value.
We show a plot of $\max_j I_2(\vec x^j_0,\vec y^j_0)$ against $\alpha$ for
$\mu=1,0.9,0.5$ in \cref{hell}. The condition $I_2(\vec x_0,\vec y_0)\le
	\lambda<1$ is satisfied with $\mu=1$ for $\alpha$ small. For larger values
of $\alpha$, $\mu$ must be reduced; for example, for $\alpha=1$,
$\mu$ must be reduced to $\mu\approx 0.5$.
\begin{figure}
	\begin{center}
		\begin{tikzpicture}
			\begin{axis}[xlabel={$\alpha$}, ylabel={$\max_j I_2$}, width=\textwidth*0.4,  legend style={at={(1.3,1.)},anchor=north}];
				\addplot+[]
				coordinates {
						(0.25,0.2) (0.5,0.6) (1.0,1.3)
						(1.5,1.53) (1.8,1.7)	 (1.9, 1.72)};
				\addlegendentry{$\mu=1$}
				\addplot+[]
				coordinates {
						(0.25,0.20) (0.5,0.48) (1.,1.1)
						(1.5,1.24) (1.8, 1.43) (1.9,1.55)};%
				\addlegendentry{$\mu=0.9$}
				\addplot+[mark=diamond]
				coordinates {
						(0.25,0.2) (0.5,0.48) (1.,0.73)
						(1.5,0.92) (1.8, 0.98) (1.9,1.2)};%
				\addlegendentry{$\mu=0.5$}
			\end{axis}
		\end{tikzpicture}
	\end{center}
	\caption{The maximum of $\max_{j=1,\dots,20} I_2(\vec x_0^j,\vec y^j_0)$
		based on computing the expectation with $10^6$ samples, for $\mu=0.5,0.9,1.0$.}\label{hell}
\end{figure}
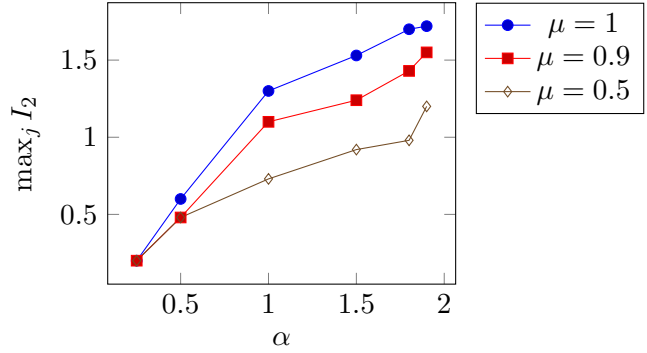

The expression for $I_1$ involves two parameters $A,t$, and we expect the
appropriate choice of $t$ to depend  on $\alpha$ and for $A\to\infty$ as
$t\to 1$. Again, we evaluate $I_1(\vec x_0)$ based on $10^6$ samples and
\cref{heaven} shows the result of $\max_j I_1(\vec x_0^j)$ for different
choice
of $A$ and $t$. By choice of $A$, we can always achieve $\max_j I_1(\vec
	x_0^j)\le \lambda<1$ for $t=0.9$.

\begin{table}
	\begin{center}
		\begin{tabular}{c  c  c r }
			$\alpha$ & $A$              & $t$            & $\max_j I_1(\vec x_0^j)$ \\
			\midrule
			0.5      & $10^4$           & 1.0            & 0.53                     \\
			1.0      & $\smash{\vdots}$ & \smash{\vdots} & 1.03                     \\
			1.5      & $\smash{\vdots}$ & 0.9            & 3.55                     \\
			1.5      & $\smash{\vdots}$ & 0.5            & 0.86                     \\
			1.8      & $\smash{\vdots}$ & \smash{\vdots} & 0.93                     \\
			1.5      & $10^5$           & 0.9            & 0.85                     \\
			1.8      & $10^6$           & \smash{\vdots} & 0.94
		\end{tabular}
	\end{center}
	\caption{The computation of $\max_{j=1,\dots,20} I_1(\vec x_0^j)$
		is based on computing the expectation with $10^6$ samples. The $\smash{\vdots}$
		symbol indicates
		the value is given by the entry above.}
	\label{heaven}
\end{table}

\bibliographystyle{alpha}
{\footnotesize\bibliography{field_wos}}
\end{document}